\RequirePackage{etoolbox}
\csdef{input@path}{{style/}{graphics/}}
\documentclass[aop,MSNbibl,nameyear,dvips]{arximspdf}
\usepackage{graphicx}
\input xypic
\doi{10.1214/13-AOP873} %kopijuoti is PTS
\volume{43}
\issue{2}
\pubyear{2015}
\firstpage{682}
\lastpage{737}

\makeatletter
\newcommand{\rrvert}{\vert}
\newcommand{\llvert}{\vert}
\newcommand{\eqref}[1]{(\ref{#1})}
\newtheorem{thmm}{Theorem}[section]
\newtheorem{prop}[thmm]{Proposition}
\newtheorem{lem}[thmm]{Lemma}

\newcommand{\A}{\mathcal{A}}
\newcommand{\jdt}{jeu de taquin\ }
\newcommand{\prob}{\operatorname{Prob}}
\newcommand{\demph}{\emph}
\newcommand{\partitions}[1]{\mathbb{Y}_{#1}}
\newcommand{\allpartitions}{\mathbb{Y}}
\newcommand{\R}{\mathbb{R}}
\newcommand{\N}{\mathbb{N}}
\newcommand{\Z}{\mathbb{Z}}
\newcommand{\C}{\mathbb{C}}
\newcommand{\E}{\mathbb{E}}

\newcommand{\convdist}{\mathop{\rightarrow}^{\mathcal{D}}} %
\newcommand{\equalindist}{\stackrel{\mathcal{D}}{=}} % notation for
\newcommand{\Young}{\partitions}
\newcommand{\jpath}{\mathbf{p}}
\newcommand{\jpathlazy}{\mathbf{q}}
\newcommand{\oproduct}{\circ}
\newcommand{\plancherel}{\mathsf{P}}
\newcommand{\projection}{q}

\newcommand{\Rec}{\operatorname{Ins}}
\newcommand{\Tr}{\operatorname{Tr}}
\newcommand{\Var}{\operatorname{Var}}
\newcommand{\RSK}{\operatorname{RSK}}
\newcommand{\SC}{\mathrm{SC}}
\newcommand{\SClaw}{\mathcal{L}_{\SC}}
\newcommand{\FSC}{F_{\SC}}
\newcommand{\Id}{\operatorname{Id}}
\newcommand{\regular}{\operatorname{regular}}
\newcommand{\trivial}{\operatorname{trivial}}
\newcommand{\End}{\operatorname{End}}
\newcommand{\SYT}{\operatorname{SYT}}
\newcommand{\strangeSchur}{\widetilde{S}}
\makeatother

\begin{document}
\begin{frontmatter}

\title{Jeu de taquin dynamics on infinite Young tableaux and
second class particles}
\runtitle{Jeu de taquin dynamics on infinite Young tableaux}

\begin{aug}
\author[A]{\fnms{Dan} \snm{Romik}\corref{}\thanksref{T2}\ead[label=e1]{romik@math.ucdavis.edu}}
\and
\author[B]{\fnms{Piotr} \snm{\'Sniady}\thanksref{T3}\ead[label=e2]{Piotr.Sniady@math.uni.wroc.pl}}
\thankstext{T2}{Supported by NSF Grant DMS-09-55584.}
\thankstext{T3}{Supported by the Polish Ministry of Higher Education
research Grant N N201 364436 for the years 2009--2012.
In the initial phase of this research, Piotr \'Sniady was a holder of a
fellowship of Alexander-von-Humboldt-Stiftung.}
\runauthor{D.~Romik and P.~\'Sniady}
\affiliation{University of California, Davis, and Polish Academy of
Sciences and University of Wroc\l{}aw}

\address[A]{Department of Mathematics\\
University of California, Davis\\
One Shields Avenue \\
Davis, California 95616\\
USA\\
\printead{e1}}

\address[B]{Institute of Mathematics\\
Polish Academy of Sciences\\
ul.~\'Sniadeckich 8\\
00-956 Warszawa\\
Poland\\
and\\
Institute of Mathematics\\
University of Wroc\l{}aw\\
pl.~Grunwaldzki~2/4\\
50-384 Wroc\l{}aw\\
Poland\\
\printead{e2}}
\end{aug}

% HISTORY:
\received{\smonth{11} \syear{2012}}
\revised{\smonth{6} \syear{2013}}

% ABSTRACT
%
\begin{abstract}
We study an infinite version of the ``\emph{jeu de taquin}'' sliding
game, which can be thought of as a natural measure-preserving
transformation on the set of infinite Young tableaux equipped with the
Plancherel probability measure. We use methods from representation
theory to show that the Robinson--Schensted--Knuth ($\RSK$)
algorithm gives an isomorphism between this measure-preserving
dynamical system and the one-sided shift dynamics on a sequence of
independent and identically distributed random variables distributed
uniformly on the unit interval. We also show that the \jdt paths
induced by the transformation are asymptotically straight lines
emanating from the origin in a random direction whose distribution is
computed explicitly, and show that this result can be interpreted as a
statement on the limiting speed of a second-class particle in the \emph
{Plancherel-TASEP} particle system (a variant of the Totally Asymmetric
Simple Exclusion Process associated with Plancherel growth), in analogy
with earlier results for second class particles in the ordinary TASEP.
\end{abstract}

% KEYWORDS

\end{frontmatter}

%s1 #&#
\section{Introduction}\label{sec1}

%s1.1 #&#
\subsection{Overview: Jeu de taquin on infinite Young tableaux}
The goal of this paper is to study in a new probabilistic framework a
combinatorial process that is well known to algebraic combinatorialists
and representation theorists. This process is known as the \emph{\jdt}
(literally ``teasing game'') or sliding game. Its remarkable properties
have been studied since its introduction in a seminal paper by \citet{Schutzenberger1977}. Its main importance is as a tool for studying the
combinatorics of permutations and Young tableaux, especially with
regards to the Robinson--Schensted--Knuth ($\RSK$) algorithm, which is
a fundamental object of algebraic combinatorics. However, the existing
\jdt theory deals exclusively with the case of \emph{finite}
permutations and tableaux. A main new idea of the current paper is to
consider the implications of ``sliding theory'' for \emph{infinite}
tableaux. As the reader will discover below, this will lead us to some
important new insights into the asymptotic theory of Young tableaux, as
well as to unexpected new connections to ergodic theory and to
well-known random processes of contemporary interest in probability
theory, namely the Totally Asymmetric Simple Exclusion Process (TASEP),
the corner growth model and directed last-passage percolation.

Our study will focus on a certain measure-preserving dynamical system,
that is, a quadruple $\mathfrak{J}=(\Omega, \mathcal{F}, \plancherel,
J)$, where $(\Omega, \mathcal{F}, \plancherel)$ is a probability space
and $J\dvtx \Omega\to\Omega$ is a measure-preserving transformation. The
sample space $\Omega$ will be the set of \demph{infinite Young
tableaux}; the probability measure $\plancherel$ will be the \demph
{Plancherel measure}, and the measure-preserving transformation $J$
will be the \demph{jeu de taquin map}. To define these concepts, we
need to first recall some basic notions from combinatorics.

%s1.2 #&#
\subsection{Basic definitions}
%s1.2.1 #&#
\subsubsection{Young diagrams and Young tableaux}
Let $n\ge1$ be an integer. An \demph{integer partition} (or just
\demph
{partition}) of $n$ is a representation of $n$ in the form $n=\lambda
(1)+\lambda(2)+\cdots+\lambda(k)$, where $\lambda(1)\ge\cdots\ge
\lambda
(k)>0$ are integers. Usually the vector $\lambda=(\lambda(1),\ldots
,\lambda(k))$ is used to denote the partition. We denote the set of
partitions of $n$ by $\partitions{n}$ (where we also define
$\partitions
{n}$ for $n=0$ as the singleton set consisting of the ``empty
partition,'' denoted by $\varnothing$), and the set of all partitions by
$\allpartitions=\bigcup_{n=0}^\infty\partitions{n}$. If $\lambda
\in
\partitions{n}$ we call $n$ the \demph{size} of $\lambda$ and denote
$|\lambda|=n$.

Given a partition $\lambda=(\lambda(1),\ldots,\lambda(k))$ of $n$, we
associate with it a \demph{Young diagram}, which is a diagram of $k$
left-justified rows of unit squares (also called boxes or cells) in
which the $j$th row has $\lambda(j)$ boxes. We use the French
convention of drawing the Young diagrams from the bottom up; see
Figure~\ref{fig-youngdiag}. Since Young diagrams are an equivalent way
of representing integer partitions, we refer to a Young diagram
interchangeably with its associated partition.

%f1 #&#
\begin{figure}

\includegraphics{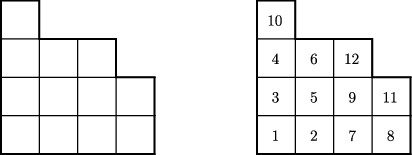}

\caption{The Young diagram $\lambda=(4,4,3,1)$ and a Young tableau of
shape $\lambda$.}
\label{fig-youngdiag}
\end{figure}

The set $\allpartitions$ of Young diagrams forms in a natural way the
vertex set of a directed graph called the \demph{Young graph} (or
\demph
{Young lattice}), where we connect two diagrams $\lambda, \nu$ by a
directed edge if $|\nu|=|\lambda|+1$ and $\nu$ can be obtained from
$\lambda$ by the addition of a single box; see Figure~\ref{fig-younggraph}. We denote the adjacency relation in this graph by
$\lambda\nearrow\nu$.

%f2 #&#
\begin{figure}[b]

\includegraphics{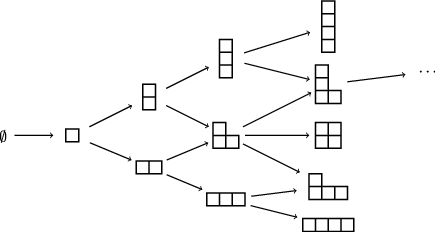}

\caption{The Young graph.}
\label{fig-younggraph}
\end{figure}

Given a Young diagram $\lambda$ of size $n$, an \demph{increasing
tableau} of shape $\lambda$ is a filling of the boxes of $\lambda$ with
some distinct real numbers $x_1,\ldots,x_n$ such that the numbers along
each row and column are in increasing order. A \demph{Young tableau}
(also called \demph{standard Young tableau} or \demph{standard
tableau}) of shape $\lambda$ is an increasing tableau of shape
$\lambda
$ where the numbers filling it are exactly $1,\ldots,n$. The set of
standard Young tableaux of shape $\lambda$ will be denoted by $\SYT
_\lambda$. One useful way of thinking about these objects is that a
Young tableau $t$ of shape $\lambda$ encodes (bijectively) a path in
the Young graph
%
%e1 #&#
\begin{equation}
\label{eq:younggraph-path} \varnothing= \lambda_0 \nearrow\lambda_1
\nearrow\cdots\nearrow \lambda_n = \lambda
\end{equation}
starting with the empty diagram and ending at $\lambda$. The way the
encoding works is that the $k$th diagram $\lambda_k$ in the path is the
Young diagram consisting of these boxes of $\lambda$ which contain a
number $\le k$. Going in the opposite direction, given the path~\eqref
{eq:younggraph-path} one can reconstruct the Young tableau by writing
the number $k$ in a given box if that box was added to $\lambda_{k-1}$
to obtain $\lambda_k$. The Young tableau $t$ constructed in this way is
referred to as the \demph{recording tableau} of the sequence \eqref
{eq:younggraph-path}.

%s1.2.2 #&#
\subsubsection{Plancherel measure}
Denote by $f^\lambda$ the number of standard Young tableaux of shape
$\lambda$. It is well known that
\[
n! = \sum_{\lambda\in\partitions{n}} \bigl(f^\lambda
\bigr)^2,
\]
a fact easily explained by the $\RSK$ algorithm [\citet{Fulton1997}, page 52].
Thus, if we define a measure $\plancherel_n$ on $\partitions{n}$ by setting
%
%e2 #&#
\begin{equation}
\label{eq:plancherel} \plancherel_n(\lambda) = \frac{(f^\lambda)^2}{n!} \qquad(\lambda
\in \partitions{n}),
\end{equation}
then $\plancherel_n$ is a probability measure. The measure
$\plancherel
_n$ is called the \demph{Plancherel measure} of order $n$. From the
viewpoint of representation theory, one can argue that this is one of
the most natural probability measures on $\partitions{n}$ since it
corresponds to taking a random irreducible component of the
left-regular representation of the symmetric group $S_n$, which is one
of the most natural and important representations;
see Section~\ref{subsec:random-young-diagram-to-representation} below.

Another well-known fact is that the Plancherel measures of all
different orders can be coupled to form a Markov chain
%
%e3 #&#
\begin{equation}
\label{eq:plancherel-growth} \varnothing= \Lambda_0 \nearrow\Lambda_1
\nearrow\Lambda_2 \nearrow \cdots,
\end{equation}
where each $\Lambda_n$ is a random Young diagram distributed according
to $\plancherel_n$. This is done by defining the conditional
distribution of $\Lambda_{n+1}$ given $\Lambda_n$ using the following
transition rule:
%
%e4 #&#
\begin{equation}
\label{eq:transition-rule} \prob(\Lambda_{n+1} = \nu | \Lambda_n =
\lambda) = %
\cases{\displaystyle\frac{f^\nu}{(n+1) f^\lambda}, & \quad$\mbox{if }\lambda \nearrow
\nu$, \vspace*{2pt}
\cr
0, &\quad $\mbox{otherwise,}$ } %
\end{equation}
for each $\lambda\in\partitions{n}, \nu\in\partitions{n+1}$. The fact
that the right-hand side of \eqref{eq:transition-rule} defines a valid
Markov transition matrix and that the push-forward of the measure
$\plancherel_n$ under this transition rule is $\plancherel_{n+1}$ is
explained by \citet{Kerov1999}, where the process $(\Lambda
_n)_{n=0}^\infty$ has been called the \demph{Plancherel growth process}
[see also \citet{Romik2013}, Section~1.19]. Here, we shall think of the
same process in a slightly different way by looking at the recording
tableau associated with the chain \eqref{eq:plancherel-growth}. Since
this is now an infinite path in the Young graph, the recording tableau
is a new kind of object which we call an \demph{infinite Young
tableau}. This is defined as an infinite matrix
$t=(t_{i,j})_{i,j=1}^\infty$ of natural numbers where each natural
number appears exactly once and the numbers along each row and column
are increasing.
Graphically, an infinite Young tableau can be visualized, similarly as
before, as a filling of the boxes of the ``infinite Young diagram''
occupying the entire first quadrant of the plane by the natural
numbers. We use the convention that the numbering of the boxes follows
the Cartesian coordinates, that is, $t_{i,j}$ is the number written in
the box $(i,j)$ which is in the $i$th column and $j$th row, with the
rows and columns numbered by the elements of the set $\N=\{1,2,\ldots\}$
of the natural numbers. Denote by $\Omega$ the set of infinite Young tableaux.

We remark that the usual (i.e., noninfinite) Young tableaux are very
useful in the \emph{representation theory of the symmetric groups}: one
can find a very natural base of the appropriate representation space
which is indexed by Young tableaux [\citet
{Ceccherini-SilbersteinScarabottiTolli2010}]. Thus, it should not come
as a surprise that infinite tableaux are very useful for studying
harmonic analysis on the \emph{infinite symmetric group} $S_\infty$;
see \citet{VershikKerov1981}.

Now, just as finite Young tableaux are in bijection with paths in the
Young graph leading up to a given Young diagram, the infinite Young
tableaux are similarly in bijection with those infinite paths in the
Young graph starting from the empty diagram that have the property that
any box is eventually included in some diagram of the path. We call an
infinite tableau corresponding to such an infinite path the \demph
{recording tableau} of the path, similarly to the case of finite paths.
Thus, under this bijection the Plancherel growth process \eqref
{eq:plancherel-growth} can be interpreted as a random infinite Young
tableau; that is, a probability measure on the set $\Omega$ of infinite
Young tableaux, equipped with its natural measurable structure, namely,
the minimal $\sigma$-algebra $\mathcal{F}$ of subsets of $\Omega$ such
that all the coordinate functions $t \mapsto t_{i,j}$ are measurable.
(Note that the Plancherel growth process almost surely has the property
of eventually filling all the boxes---for example, this follows
trivially from Theorem~\ref{thmm-plancherel-limitshape} below.)

We denote this probability measure on $(\Omega, \mathcal{F})$ by
$\plancherel$, and refer to it as the \demph{Plancherel measure of
infinite order}, or (where there is no risk of confusion) simply \demph
{Plancherel measure}.

%s1.2.3 #&#
\subsubsection{Jeu de taquin}
Given an infinite Young tableau $t=(t_{i,j})_{i,j=1}^\infty\in\Omega$,
define inductively an infinite up-right lattice path in $\mathbb{N}^2$
%
%e5 #&#
\begin{equation}
\label{eq:jdt-path} \jpath_1(t), \jpath_2(t),
\jpath_3(t), \ldots,
\end{equation}
where $\jpath_1(t)=(1,1)$, and for each $k\ge2$, $\jpath_k =
(i_k,j_k)$ is given by
%
%e6 #&#
\begin{equation}
\label{eq:jdt-path-def} \jpath_k = %
\cases{ (i_{k-1}+1,j_{k-1}),
&\quad $\mbox{if } t_{i_{k-1}+1,j_{k-1}} < t_{i_{k-1},j_{k-1}+1}$, \vspace*{2pt}
\cr
(i_{k-1},j_{k-1}+1), &\quad $\mbox{if } t_{i_{k-1}+1,j_{k-1}} >
t_{i_{k-1},j_{k-1}+1}$.} %
\end{equation}
That is, one starts from the corner box of the tableau and starts
traveling in unit steps to the right and up, at each step choosing the
direction among the two in which the entry in the tableau is smaller.
We refer to the path \eqref{eq:jdt-path} defined in this way as the
\demph{jeu de taquin path} of the
tableau $t$. This is illustrated in Figure~\ref{fig-jdt-path}(a).

%f3 #&#
\begin{figure}

\includegraphics{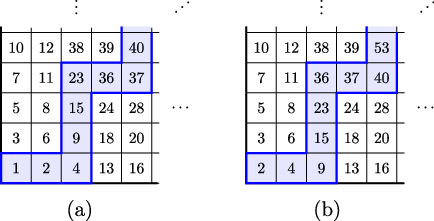}

\caption{\textup{(a)} A part of an infinite Young
tableau $t$. The highlighted boxes form the beginning of the \jdt path
$\jpath(t)$. \textup{(b)} The outcome of ``sliding'' of
the boxes along the highlighted \jdt path. The outcome of the \jdt
transformation $J(t)$ is obtained by subtracting $1$ from all the entries.}
\label{fig-jdt-path}
\end{figure}

%
%}
%}
%

We now use the jeu de taquin path to define a new infinite tableau
$s=J(t) =(s_{i,j})_{i,j=1}^\infty$, using the formula
%
%e7 #&#
\begin{equation}
\label{eq:jdt-transformation} s_{i,j} = %
\cases{t_{\jpath_{k+1}}-1, & \quad $
\mbox{if } (i,j)=\jpath_k\mbox{ for some }k$, \vspace*{2pt}
\cr
t_{i,j}-1, & \quad $\mbox{otherwise.}$} %
\end{equation}
The mapping $t\mapsto s=J(t)$ defines a transformation $J\dvtx \Omega\to
\Omega$, which we call the \demph{jeu de taquin map}. In words, the way
the transformation works is by removing the box at the corner, then
sliding the second box of the jeu de taquin path into the space left
vacant by the removal of the first box, and continuing in this way,
successively sliding each box along the jeu de taquin path into the
space vacated by its predecessor.
At the end, one subtracts $1$ from all entries to obtain a new array of
numbers. It is easy to see that the resulting array is an infinite
Young tableau: the definition of the \jdt path guarantees that the
sliding is done in such a way that preserves monotonicity along rows
and columns. For an example, compare Figure \ref{fig-jdt-path}(a) and
\ref{fig-jdt-path}(b).

The above construction is a generalization of the construction of\break 
\citet{Schutzenberger1977} who introduced it for finite Young tableaux.\break  Sch\"
utzenberger's jeu de taquin turned out to be a very powerful tool
of algebraic combinatorics and the representation theory of symmetric
groups; in particular, it is important in studying combinatorics of
words, the Robinson--Schensted--Knuth ($\RSK$) correspondence and the
Littlewood--Richardson rule; see \citet{Fulton1997} for an overview.

%s1.2.4 #&#
\subsubsection{An infinite version of the Robinson--Schensted--Knuth algorithm}
\label{subsubsec:infinite-rsk}

Next, we consider an infinite version of the \emph{Robinson--Schensted--Knuth\break (RSK) algorithm} which can be applied to an
infinite sequence $(x_1,x_2,x_3,\ldots)$ of distinct real
numbers.\setcounter{footnote}{2}\footnote{Actually, this is an infinite version of a special
case of $\RSK$ that predates it and is known as the \demph
{Robinson--Schensted algorithm}, but we prefer to use the $\RSK$
mnemonic due to its convenience and familiarity to a large number of
readers.} This infinite version was considered in a more general setup
by \citet{KerovVershik1986} [the finite version of the algorithm,
summarized here, is discussed in detail by \citet{Fulton1997}]. The
algorithm performs an inductive computation, reading the inputs
$x_1,x_2, \ldots$ successively, and at each step applying a so-called
\emph{insertion step} to its previous computed output together with the
next input $x_n$.

The insertion step, given an increasing tableau $P_{n-1}$ and a number
$x_n$ produces a new increasing tableau $P_n$ whose shape $\lambda_n$
is obtained from $\lambda_{n-1}$ by the addition of a single box. The
new tableau $P_n$ is computed by performing a succession of \emph
{bumping steps} whereby $x_n$ is inserted into the first row of the
diagram (as far to the right as possible so that the row remains
increasing and no gaps are created), bumping an existing entry from the
first row into the second row, which results in an entry of the second
row being bumped to the third row, and so on, until finally the entry
being bumped settles down in an unoccupied position outside the diagram~$\lambda$. An example is shown in Figure~\ref{fig:insertion-step}.

%f4 #&#
\begin{figure}

\includegraphics{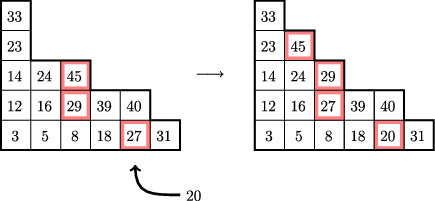}

\caption{Example of an insertion step. The highlighted boxes indicate
the locations of bumped entries.}
\label{fig:insertion-step}
\end{figure}

For each $n\ge0$, after inserting the first $n$ inputs $x_1,\ldots
,x_n$ the algorithm produces a triple $(\lambda_n, P_n, Q_n)$, where
$\lambda_n \in\partitions{n}$ is a Young diagram with $n$ boxes, $P_n$
is an increasing tableau of shape $\lambda_n$ containing the numbers
$x_1, \ldots, x_n$, and $Q_n$ is a standard Young tableau of shape
$\lambda_n$. The shapes satisfy $\lambda_{n-1}\nearrow\lambda_n$, that
is, at each step one new box is added to the current shape, with the
tableau $Q_n$ being simply the recording tableau of the path $\varnothing
= \lambda_0 \nearrow\lambda_1 \nearrow\ldots\nearrow\lambda_n$. The
tableau $P_n$ is the information that will be acted upon by the next
insertion step, and is called the \demph{insertion tableau}.
We will refer to $\lambda_n$ as the \demph{$\RSK$ shape associated to
$(x_1,\ldots,x_n)$}.

In this infinite version of the algorithm, we shall assume that
$x_1,x_2,\ldots$ are such that the infinite Young graph path
$\varnothing
= \lambda_0 \nearrow\lambda_1 \nearrow\ldots$ can be encoded by an
infinite recording tableau $Q_\infty$ (i.e., we assume that every box
in the first quadrant eventually gets added to some $\lambda_n$). For
our purposes, the information in the insertion tableaux $P_n$ will not
be needed, so we simply discard it, and define the \demph{\textup{(}infinite\textup{)}
$\RSK$ map} by
\[
\RSK(x_1,x_2,\ldots)=Q_\infty.
\]

%s1.3 #&#
\subsection{The main results}

We are now ready to state our main results.

%s1.3.1 #&#
\subsubsection{The jeu de taquin path}
Our first result concerns the asymptotic behavior of the jeu de taquin path.
For a given infinite tableau $t\in\Omega$, we define $\Theta= \Theta
(t)\in[0,\pi/2 ]$ by
\[
\bigl(\cos\Theta(t),\sin\Theta(t) \bigr) = \lim_{k\to\infty}
\frac
{\jpath_k(t)}{\Vert\jpath_k(t)\Vert}
\]
whenever the limit exists, and in this case refer to $\Theta$ as the
\demph{asymptotic angle} of the \jdt path.

%th1.1 #&#
\begin{thmm}[(Asymptotic behavior of the jeu de taquin path)]
\label{thmm-straight-line}
The jeu de taquin path converges $\plancherel$-almost surely to a
straight line with a random direction. More precisely, we have
\[
\plancherel \biggl[\lim_{k\to\infty} \frac{\jpath_k}{\Vert\jpath
_k\Vert
}\mbox{ exists }
\biggr] = 1.
\]
Under the Plancherel measure $\plancherel$, the asymptotic angle
$\Theta
$ is an absolutely continuous random variable on $[0,\pi/2]$ whose
distribution has the following explicit description:
%
%e8 #&#
\begin{equation}
\label{eq:theta-dist} \Theta\equalindist\Pi(W),
\end{equation}
where $W$ is a random variable distributed according to the semicircle
distribution $\SClaw$ on $[-2,2]$, that is, having density given by
%
%e9 #&#
\begin{equation}
\label{eq:semicircle} \SClaw(dw) = \frac{1}{2 \pi} \sqrt{4-w^2} \,dw\qquad \bigl (|w|
\le2\bigr),
\end{equation}
and $\Pi(\cdot)$ is the function
\[
\Pi(w) = \frac{\pi}{4}-\cot^{-1} \biggl[\frac{2}{\pi} \biggl(
\sin ^{-1} \biggl(\frac{w}{2} \biggr)+\frac{\sqrt{4-w^2}}{w} \biggr)
\biggr]\qquad (-2 \le w \le2).
\]
\end{thmm}

Figure~\ref{fig:theta} shows simulation results illustrating the
theorem. Figure~\ref{fig-theta-density-plot} shows a plot of the
density function of $\Theta$. Note that the definition of the
distribution of $\Theta$ has a more intuitive geometric description;
see Section~\ref{sec-weak-asym} for the details.

%f5 #&#
\begin{figure}

\includegraphics{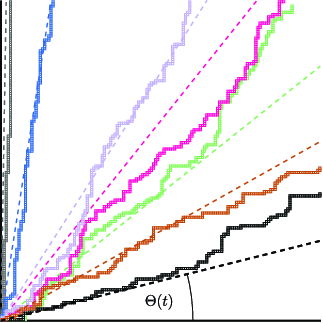}

\caption{Several simulated paths of \jdt and (dashed lines) their asymptotes.}
\label{fig:theta}
\end{figure}

%f6 #&#
\begin{figure}[b]

\includegraphics{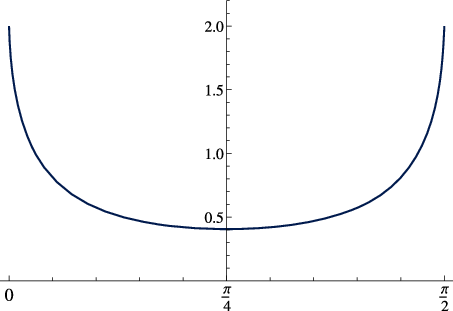}

\caption{A plot of the density function of $\Theta$. The density is
bounded but is heavily skewed, with most of the probability
concentrated near the ends of the interval $[0,\pi/2]$.}
\label{fig-theta-density-plot}
\end{figure}

%s1.3.2 #&#
\subsubsection{The Plancherel-TASEP interacting particle system}

One topic that we will explore in more detail later is an analogy between
Theorem~\ref{thmm-straight-line} and a result of \citet
{FerrariPimentel2005} on competition interfaces in the corner growth
model. Furthermore, this result is essentially a reformulation of
previous results of \citet{FerrariKipnis1995} and \citet
{MountfordGuiol2005} on the limiting speed of second class particles in
the \demph{Totally Asymmetric Simple Exclusion Process (TASEP)};
similarly, our Theorem~\ref{thmm-straight-line} affords a
reinterpretation in the language of interacting particle systems,
involving a variant of the TASEP which we call the \demph
{Plancherel-TASEP particle system}. We find this reinterpretation to be
just as interesting as the result above. However, because of the
complexity of the necessary background, and to avoid making this
introductory section excessively long, we formulate this version of the
result here without explaining the meaning of the terminology used, and
defer the details and further exploration of this connection to
Section~\ref{sec-particle-systems}. We encourage the reader to visit
the discussion in that section to gain a better appreciation of the
context and importance of the result.

%th1.2 #&#
\begin{thmm}[(The second class particle trajectory)] \label
{thmm-second-class-plancherel-tasep}
For $n\ge0$, let $X(n)$ denote the location at time $n$ of the
second-class particle in the\break Plancherel-TASEP interacting particle system.
The limit
\[
W = \lim_{n\to\infty} \frac{X(n)}{\sqrt{n}}
\]
exists almost surely and is a random variable distributed according to
the semicircle distribution $\SClaw$.
\end{thmm}

The limiting random variable $W$ can be thought of as an asymptotic
speed parameter for the second-class particle. Namely, if one considers
for each $n\ge1$ the scaled trajectory functions
%
%e10 #&#
\begin{equation}
\label{eq:second-class-rescaled} \widehat{X}_n(t) =
 \frac{X(\lfloor n t \rfloor)}{\sqrt{n}}\qquad (t>0),
\end{equation}
then Theorem~\ref{thmm-second-class-plancherel-tasep} can be
reformulated as saying that as $n\to\infty$, almost surely the
trajectory will follow asymptotically one of the curves in the
one-parameter family
$(\alpha\sqrt{t})_{-2\le\alpha\le2}$, where the parameter $\alpha
$ is
random and chosen according to the distribution $\SClaw$. If one
reparameterizes time by replacing $t$ with $t^2$ (which is arguably a
more natural parameterization---see the discussion in
Section~\ref{sec:stochastic}), we get the statement that the limiting
trajectory of the second-class particle is asymptotically a straight
line with slope $\alpha$. This is analogous to the result of \citet
{MountfordGuiol2005}, where the process is the ordinary TASEP and the
limiting speed of the second-class particle has the uniform
distribution $U(-1,1)$ on the interval $[-1,1]$.

%s1.3.3 #&#
\subsubsection{The jeu de taquin dynamical system}

It is worth pointing out that the \jdt applied to an infinite tableau
$t\in\Omega$ produces two interesting pieces of information: the \jdt
path \eqref{eq:jdt-path}:
\[
\jpath(t)= \bigl( \jpath_1(t), \jpath_2(t), \ldots
\bigr),
\]
and another infinite tableau $J(t)\in\Omega$.
This setup naturally raises questions about the iterations of the \jdt map
\[
t, J(t), J \bigl(J(t) \bigr), \ldots
\]
or, in other words, about the dynamical system $\mathfrak{J}=(\Omega,
\mathcal{F}, \plancherel, J)$, which we call the \demph{\jdt dynamical
system}. The following result shows that this is indeed a very natural
point of view.

%th1.3 #&#
\begin{thmm}[(Measure preservation and ergodicity)]
\label{thmm-measure-preservation}
The dynamical system $\mathfrak{J}=(\Omega, \mathcal{F}, \plancherel
,J)$ is measure-preserving and ergodic.
\end{thmm}

We believe the part of the above result concerning the measure-preservation may be known to experts in the field, though we are not
aware of a reference to it in print. The second part concerning
ergodicity is new.

The next result sheds light on the behavior of the \jdt dynamical
system~$\mathfrak{J}$, by showing that it has probably the simplest
possible structure one could hope for, namely, it is isomorphic to an
i.i.d. shift.

%th1.4 #&#
\begin{thmm}[(Isomorphism to an i.i.d. shift map)]
\label{thmm-isomorphism}
Let\/ $\mathfrak{S}=([0,1]^\mathbb{N},\mathcal{B},\break  \mathrm
{Leb}^{\otimes\mathbb{N}},S)$ denote
the measure-preserving dynamical system corresponding to the
(one-sided) shift map on an infinite sequence of independent random
variables with the uniform distribution $U(0,1)$ on the unit interval
$[0,1]$. That is, $\mathrm{Leb}^{\otimes\mathbb{N}}=\prod_{n=1}^\infty(\operatorname{Leb})$
is the product of Lebesgue measures on $[0,1]$, $\mathcal{B}$ is the
product $\sigma$-algebra on $[0,1]^\mathbb{N}$, and $S\dvtx [0,1]^\mathbb
{N}\to[0,1]^\mathbb{N}$ is the shift map, defined by
\[
S(x_1,x_2,\ldots) = (x_2,x_3,
\ldots).
\]
Then the mapping $\RSK\dvtx [0,1]^\mathbb{N}\to\Omega$ is an isomorphism
between the measure-preserving dynamical systems $\mathfrak{J}$ and
$\mathfrak{S}$.
\end{thmm}

Note that such a complete characterization of the highly nontrivial
measure-preserving system $\mathfrak{J}$ may open up many
possibilities for additional applications. We hope to explore these
possibilities in future work. Furthermore, in contrast to many
structure theorems in ergodic theory that show isomorphism of
complicated dynamical systems to i.i.d. shift maps via an abstract
existential argument that does not provide much insight into the nature
of the isomorphism, here the isomorphism is a completely explicit,
familiar and highly structured mapping---the $\RSK$ algorithm.

Note also that $\RSK$ is defined on the set of sequences
$(x_1,x_2,\ldots)$ which satisfy the assumption mentioned in Section~\ref{subsubsec:infinite-rsk}. This is known (see again Theorem~\ref
{thmm-plancherel-limitshape} below) to be a set of full measure with
respect to $\mathrm{Leb}^{\otimes\mathbb{N}}$.

Theorem~\ref{thmm-isomorphism} above encapsulates several separate
claims: first, that the Plancherel measure $\plancherel$ is the
push-forward of the product measure $\mathrm{Leb}^{\otimes\mathbb
{N}}$ under the mapping
$\RSK
$; this is easy and well known (see Lemma~\ref{lem-push-forward}
below). Second, that $\RSK$ is a factor map (also known as
homomorphism) of measure-preserving dynamical systems. This is the
statement that
%
%e11 #&#
\begin{equation}
\label{eq:factormap} J\circ\RSK= \RSK\circ\, S,
\end{equation}
that is, the following diagram commutes:
\[
\xymatrix{
{[0,1]^\mathbb{N}} \ar[r]^{S}\ar[d]^{\mathrm{RSK}}& {[0,1]^\mathbb{N}} \ar[d]^{\mathrm{RSK}}\\
\Omega\ar[r]^{J}& \Omega}
\]
This is somewhat nontrivial but follows from known combinatorial
properties of the $\RSK$ algorithm and jeu de taquin in the finite
setting. Finally, the hardest part is the claim that this factor map is
in fact an isomorphism. It is also the most surprising: recall that in
the infinite version of the $\RSK$ map we discarded all the information
contained in the insertion tableaux $(P_n)_{n=1}^\infty$. In the finite
version of $\RSK$, the insertion tableau is essential to inverting the
map, so how can we hope to invert the infinite version without this
information? It turns out that Theorem~\ref{thmm-straight-line} plays
an essential part: the asymptotic direction of the \jdt path provides
the key to inverting $\RSK$ in our ``infinite'' setting. This is
explained next.

%s1.3.4 #&#
\subsubsection{The inverse of infinite $\operatorname{RSK}$}

The secret to inversion of infinite $\RSK$ is as follows. We will show
in a later section (see Theorem~\ref{thmm-asym-det-jdt} below) that the
limiting direction $\Theta$ of the \jdt path is a function of only the
first input $X_1$ in the sequence of i.i.d. uniform random variables
$X_1, X_2, \ldots$ to which the $\RSK$ factor map is applied. Moreover,
this function is an explicit (and invertible) function. This gives us
the key to inverting the map $\RSK(\cdot)$ and, therefore, proving the
isomorphism claim, since, if we can recover $X_1$ from the infinite
tableau $T$, then by iterating the map $J$ and using the factor
property we can similarly recover the successive inputs $X_2, X_3,\ldots
,$ etc. Thus, we get the following explicit description of the inverse
$\RSK$ map.

%th1.5 #&#
\begin{thmm}[(The inverse of infinite $\RSK$)]
\label{thmm-inverse-rsk}
The inverse mapping\break $\RSK^{-1}\dvtx  \Omega\to[0,1]^\mathbb
{N}$ is
given $\plancherel$-almost surely by
\[
\RSK^{-1}(t) = \bigl[ F_\Theta \bigl(\Theta_1(t)
\bigr),F_\Theta \bigl( \Theta _2(t) \bigr),
F_\Theta \bigl(\Theta_3(t) \bigr), \ldots \bigr],
\]
where we denote $\Theta_k = \Theta\circ J^{k-1}$ (this refers to
functional iteration of $J$ with itself $k-1$ times), and where
$F_\Theta(s)= \plancherel ( \Theta\le s  )$ is the cumulative
distribution function of the asymptotic angle $\Theta$.
\end{thmm}

Note that one particular consequence of this theorem, which taken on
its own, already makes for a rather striking statement, is the fact
that under the measure $\plancherel$, the sequence of asymptotic angles
$(\Theta_k)_{k=1}^\infty$ obtained by iteration of the map $J$ as above
is a sequence of independent and identically distributed random
variables. The full statement of the theorem can be interpreted as the
stronger fact, which seems all the more surprising, that this
i.i.d. sequence is actually related in a simple way (via
coordinate-wise application of the monotone increasing function
$F_\Theta^{-1}$) to the original sequence of i.i.d. $U(0,1)$ random
variables fed as input to the $\RSK$ algorithm. As a referee pointed
out to us, an earlier clue to this type of isomorphism phenomenon can
be found in the context of $\RSK$ applied to random words over a finite
alphabet; see \citet{OConnell2003,OConnellYor2002} and the remark in
Section~\ref{subsec-random-words}.

%s1.4 #&#
\subsection{Overview of the paper}
We have described our main results, but the rest of the paper also
contains additional results of independent interest. The plan of the
paper is as follows. In Section~\ref{sec-jdt-elementary}, we recall
some additional facts from the combinatorics of Young tableaux, which
we use to pick some of the low-hanging fruit in our theory of infinite
jeu de taquin, namely,
the proof of Theorem~\ref{thmm-measure-preservation} and
the fact that $\RSK$ is a factor map, and as preparation for the more
difficult proofs. In Section~\ref{sec:limit-shape}, we prove a weaker
version of Theorem~\ref{thmm-straight-line} that shows convergence in
distribution (instead of almost sure convergence) of the direction of
the \jdt path to the correct distribution. This provides additional
intuition and motivation, since this weaker result is much easier to
prove than Theorem~\ref{thmm-straight-line}.

Next, we attack Theorem~\ref{thmm-straight-line}, which conceptually is
the most difficult part of the paper. Here, we apply methods from the
representation theory of the symmetric group. The necessary background
is developed in Section~\ref{sec:representation-voodoo}, where a key
technical result is proved (this is the only part of the paper where
representation theory is used, and it may be skipped if one is willing
to assume the validity of this technical result). This result is used
in Section~\ref{sec-determinism} to prove two additional results which
are of independent interest (especially to readers interested in
asymptotic properties of random Young tableaux) but which we did not
elaborate on in this \hyperref[sec1]{Introduction}. We refer to these results as the
\demph{asymptotic determinism of $\RSK$} and \demph{asymptotic
determinism of jeu de taquin}.

With the help of these results, Theorems \ref{thmm-straight-line},
\ref
{thmm-isomorphism} and \ref{thmm-inverse-rsk} are then proved in
Section~\ref{sec:proof-main-thms}.

Section~\ref{sec-particle-systems} is then dedicated to exploring the
connection between our results and the theory of interacting particle
systems. In particular, we study in depth the point of view in which a
``lazy'' version of the \jdt path is reinterpreted as encoding the
trajectory of a second-class particle in the Plancherel-TASEP particle
system, and consider how our results are analogous to results discussed
in the papers of \citet{FerrariKipnis1995,
MountfordGuiol2005,FerrariPimentel2005} in connection with the TASEP
and the closely related \demph{corner growth model} (also known under
the name \demph{directed last passage percolation}). This analogy is
one of the main ``inspirational'' forces of the paper, so the reader
interested in this point of view may want to read this section before
the more technical proofs in the sections preceding it.

Finally, Section~\ref{sec-additional-directions} mentions some
additional directions related to the ideas explored in this paper that
we plan to discuss in future work.

%s1.5 #&#
\subsection{Notation}

Throughout the paper, we use the following notational conventions: the
letters $\mu, \lambda, \nu$ will generally be used to denote
deterministic Young diagrams, and capital Greek letters such as
$\Lambda
, \Pi$ will be used to denote random Young diagrams. Similarly, lower
case letters such as $t, s$ may be used to denote a deterministic Young
tableau, and $T$ will denote a random one. The normalized semicircle
distribution \eqref{eq:semicircle} (on $[-2,2]$, which is the case when
its variance is $1$ and its even moments are the Catalan numbers) will
always be denoted by $\SClaw$. A~generic context-dependent probability
will be denoted by $\prob(\cdot)$, and expectation by $\E$; the symbol
$\plancherel$ will be reserved for Plancherel measure on the space~$\Omega$ of infinite Young tableaux. Other notation will be introduced
as needed in the appropriate place.

%s2 #&#
\section{Elementary properties of jeu de taquin and $\operatorname{RSK}$}
\label{sec-jdt-elementary}

In this section, we recall some standard facts about Young tableaux,
and use them to prove the easier parts of the results described in the
introduction (measure preservation, ergodicity and the factor map
property). We also start building some additional machinery that will
be used later to attack the more difficult claims about the asymptotics
of the \jdt path and the invertibility of $\RSK$.

%s2.1 #&#
\subsection{Finite version of jeu de taquin}
Let $\lambda\in\partitions{n}$ for some $n\ge1$. To each Young tableau
$t$ of shape $\lambda$, there is associated a finite \jdt path
$(1,1)=\jpath_1, \jpath_2, \ldots, \jpath_m$ defined analogously to
\eqref{eq:jdt-path-def} except that the path terminates at the last
place it visits inside the diagram $\lambda$, and for the purposes of
interpreting the formula \eqref{eq:jdt-path-def} we consider
$t_{i,j}=\infty$ for positions outside $\lambda$. We can similarly
define a finite \jdt map $j$ that takes a tableau $t$ of shape $\lambda
$ and returns a tableau $s=j(t)$ of shape $\mu$ for some $\mu\in
\partitions{n-1}$ satisfying $\mu\nearrow\lambda$. This is defined by
the same formula as \eqref{eq:jdt-transformation}, with the shape $\mu$
being formed from $\lambda$ by removing the last box of the \jdt path.

%le2.1 #&#
\begin{lem}
\label{lem:jdt-is-a-bijection}
For any $\lambda\in\partitions{n}$, denote by
\[
j_\lambda\dvtx \SYT_\lambda\rightarrow\bigsqcup
_{\mu: \mu\nearrow
\lambda} \SYT_\mu
\]
the restriction of the finite \jdt map $j$ to $\SYT_\lambda$. Then
$j_\lambda$
is a bijection.
\end{lem}
\begin{pf}
This is a standard fact; see \citet{Fulton1997}, page 14. The idea is
that given the tableau $s=j(t)$ and the shape $\lambda$, one can
recover $t$ by performing a ``reverse sliding'' operation,
starting from the unique cell in the difference $\lambda\setminus\mu$.
\end{pf}

From the lemma, it follows that the preimage $j^{-1}(t)$ of a tableau
$t$ of shape $\lambda$ contains one element for each $\nu$ for which
$\lambda\nearrow\nu$, namely
%
%e12 #&#
\begin{equation}
\label{eq:finite-jdt-preimage} j^{-1}(t) = \bigl\{ j_\nu^{-1}(t)
\dvtx \nu\in\allpartitions, \lambda \nearrow\nu \bigr\}.
\end{equation}

%s2.2 #&#
\subsection{Measure preservation}
\label{subsec-proof-measure-pres}

We now prove that the \jdt map $J$ preserves the Plancherel measure
$\plancherel$, which is the easier part of Theorem~\ref
{thmm-measure-preservation}. The proof requires verifying that the identity
%
%e13 #&#
\begin{equation}
\label{eq:measure-preservation} \plancherel \bigl(J^{-1}(E) \bigr) = \plancherel(E)
\end{equation}
holds for any event $E \in\mathcal{F}$. We shall do this for a family
of cylinder sets of a certain form, defined as follows. If $\lambda
=(\lambda(1),\ldots,\lambda(k))\in\partitions{n}$ and
$s=(s_{i,j})_{1\le
i\le k, 1\le j\le\lambda(i)}$ is a Young tableau of shape $\lambda$
[where $s_{i,j}$ is our notation for the entry written in the box in
position $(i,j)$], we define the event $E_s\in\mathcal{F}$ by
%
%e14 #&#
\begin{equation}
\label{eq:elementary-event} E_s = \bigl\{ t=(t_{i,j})_{i,j=1}^\infty
\in\Omega | t_{i,j} = s_{i,j}\mbox{ for all }1\le i\le k, 1\le
j \le\lambda(i) \bigr\}.
\end{equation}
The family of sets of the form $E_s$ clearly generates $\mathcal{F}$
and is a $\pi$-system, so by a standard fact from measure theory
[\citet{Durrett2010}, Theorem A.1.5, page 345], it will be enough to check
that \eqref{eq:measure-preservation} holds for $E_s$.

Note that if $s$ is the recording tableau of the path $\varnothing
=\lambda
_0\nearrow\lambda_1\nearrow\ldots\nearrow\lambda_n=\lambda$ in the
Young graph, then in the language of the Plancherel growth process
\eqref{eq:plancherel-growth}, $E_s$ corresponds to the event that
\[
\{ \Lambda_k = \lambda_k \mbox{ for }0\le k\le n\}.
\]
Therefore, it is easy to see from \eqref{eq:transition-rule} that
%
%e15 #&#
\begin{equation}
\label{eq:prob-elementary-set} \plancherel(E_s) = \frac{f^\lambda}{n!},
\end{equation}
since when multiplying out the transition probabilities in \eqref
{eq:transition-rule} one gets a telescoping product.

On the other hand, let us compute $\plancherel (J^{-1}(E_s) )$.
From \eqref{eq:finite-jdt-preimage}, we see that $J^{-1}(E_s)$ can be
decomposed as the disjoint union
\[
J^{-1}(E_s) = \bigsqcup_{\nu: \lambda\nearrow\nu}
E_{j_\nu^{-1}(s)}.
\]
Applying \eqref{eq:prob-elementary-set} to each summand, we see that
\[
\plancherel \bigl(J^{-1}(E_s) \bigr) = \sum
_{\nu\in\partitions{n+1},
\lambda
\nearrow\nu} \frac{f^\nu}{(n+1)!},
\]
and this is equal to $f^\lambda/n! = \plancherel(E_s)$ by the
well-known relation
\[
(n+1)f^\lambda= \sum_{\nu \dvtx \lambda\nearrow\nu} f^\nu
\]
[see equation (7) in \citet{GreeneNijenhuisWilf1984}; note that this
relation also explains why \eqref{eq:transition-rule} is a valid Markov
transition rule]. So, \eqref{eq:measure-preservation} holds for the
event $E_s$, as claimed.

%s2.3 #&#
\subsection{$\operatorname{RSK}$ and Plancherel measure}
The following lemma is well known [see,
e.g., \citet{KerovVershik1986}], and can be used as an equivalent alternative
definition of Plancherel measure. We include its proof for completeness.

%le2.2 #&#
\begin{lem}
\label{lem-push-forward}
Let $X_1, X_2, \ldots$ be a sequence of independent and identically
distributed random variables with the $U(0,1)$ distribution. The random
infinite Young tableau
\[
T = \RSK(X_1,X_2,\ldots)
\]
is distributed according to the Plancherel measure $\plancherel$. In
other words, $\plancherel$ is the push-forward of the product measure
$\mathrm{Leb}^{\otimes\mathbb{N}}$ (defined in Theorem~\ref
{thmm-isomorphism}) under the
mapping $\RSK\dvtx [0,1]^\mathbb{N}\to\Omega$.
\end{lem}

\begin{pf} Let $\plancherel'$ be the distribution measure of $T$.
Let $\lambda=(\lambda(1),\ldots,\lambda(k))\in
\partitions
{n}$ for some $n\ge1$ and let $s=(s_{i,j})_{1\le i\le k, 1\le j\le
\lambda(i)}$ be a Young tableau of shape~$\lambda$. Then the event $\{
T\in E_s\}$ [with $E_s$ as in \eqref{eq:elementary-event}] can be
written equivalently as $\{ Q_n = s \}$, where $Q_n$ is the recording
tableau part of the $\RSK$ algorithm output $(P_n,Q_n)$ corresponding
to the first $n$ inputs $(X_1,\ldots,X_n)$. Note that $Q_n$ is
dependent only on the order structure of the sequence $X_1,\ldots,X_n$;
this order is a uniformly random permutation in the symmetric group
$S_n$, and by the properties of the $\RSK$ correspondence,
%
%e16 #&#
\begin{equation}
\label{eq:probability-plancherel-tableau} \prob(Q_n=s) = f^\lambda/n!,
\end{equation}
since there are $f^\lambda$ possibilities to choose the insertion
tableau $P_n$, each of them corresponding to a single permutation among
the $n!$ possibilities.
Therefore, we have that
\[
\plancherel'(E_s) = \prob(T\in E_s) =
\prob( Q_n = s ) = \frac
{f^\lambda
}{n!} = \plancherel(E_s).
\]
Since this is true for any Young tableau $s$, and the events $E_s$ form
a $\pi$-system generating $\mathcal{F}$, it follows that the
measures $\plancherel'$ and $\plancherel$ coincide.
\end{pf}

%s2.4 #&#
\subsection{$\operatorname{RSK}$ is a factor map}
\label{subsec-factor-map}

We now prove \eqref{eq:factormap}. We need the following result which
concerns $\RSK$ and \jdt in the finite setup; see \citet{Sagan2001}, Proposition~3.9.3, for a proof.

%le2.3 #&#
\begin{lem}[{[\citet{Schutzenberger1963}]}]
\label{lem:factor-map}
Let $x_1,\ldots,x_n$ be distinct numbers. Let $Q_n$ be the recording
tableau associated by $\RSK$ to $(x_1,x_2,\ldots,x_n)$ and let
$\widetilde{Q}_{n-1}$ be the recording tableau associated to
$(x_2,x_3,\ldots,x_n)$. Then
\[
\widetilde{Q}_{n-1} = j(Q_n),
\]
where $j$ is the finite version of the \jdt map.
\end{lem}

Let $(x_1,x_2,\ldots) \in[0,1]^{\mathbb{N}}$ be a sequence for which
the infinite tableau\break $\RSK(x_1, x_2,\ldots)=Q_\infty$ is defined. In the
notation of the lemma, $Q_\infty$ is the unique infinite tableau that
``projects down'' to the sequence of finite recording tableaux $Q_n$
(in the sense that deleting all entries $>n$ gives $Q_n$). The sequence
of recording tableaux $\widetilde{Q}_{n-1}=j(Q_n)$ of $(x_2,\ldots
,x_n)$ for $n\ge1$ also determines a unique infinite tableau
$\widetilde
{Q}_\infty$ with the same projection property, which is therefore the
recording tableau of $(x_2,x_3,\ldots)=S(x_1,x_2,\ldots)$. Because $j$
is a finite version of $J$, it is easy to see that this implies
$J(Q_\infty)=\widetilde{Q}_\infty$, which is the relation \eqref
{eq:factormap} for the input $(x_1,x_2,\ldots)$.

Note that \eqref{eq:factormap} also implies that the measure-preserving
system $\mathfrak{J}$ is ergodic, since a factor of an ergodic system
is ergodic [\citet{Silva2008}, page 119].
So, we have finished proving Theorem~\ref{thmm-measure-preservation}.

%s2.5 #&#
\subsection{Monotonicity properties of $\operatorname{RSK}$}

We will identify the set of boxes of an infinite Young tableau with $\N
^2$. We introduce a partial order on $\N^2$ as follows:
\[
(x_1,y_1)\preceq(x_2,y_2) \quad\iff\quad
x_1 \leq x_2\mbox{ and }y_1\geq
y_2.
\]

If $\mathbf{a}=(a_1,\ldots,a_n)$ and $\mathbf{b}=(b_1,\ldots,b_k)$ are
finite sequences we denote by
\[
\mathbf{a} \mathbf{b}=(a_1,\ldots,a_n,b_1,
\ldots,b_k)
\]
their concatenation. Also, if $b$ is a number we denote by
\[
\mathbf{a} b=(a_1,\ldots,a_n,b)
\]
the sequence $\mathbf{a}$ appended by $b$, etc.

For a finite sequence $\mathbf{a}=(a_1,\ldots,a_n)$ we denote by $\Rec
(\mathbf{a})\in\N^2$ the last box which was inserted to the Young
diagram by the $\RSK$ algorithm applied to the sequence $\mathbf{a}$.
In other words, it is the box containing the biggest number in the
recording tableau associated to $\mathbf{a}$.

%le2.4 #&#
\begin{lem}
\label{lem:RSK-monotone1}
Assume that the elements of the sequence $\mathbf{a}=(a_1,\ldots,a_l)$
and $b,b'$ are distinct numbers and $b<b'$.
Then we have the relations:
\begin{longlist}[(a)]%[label=\textnormal{(\alph*)}]
\item[(a)] $\Rec(\mathbf{a} b) \prec\Rec(\mathbf{a} b b')$;
\item[(b)] $\Rec(\mathbf{a} b') \succ\Rec(\mathbf{a}
b' b)$;
\item[(c)] $\Rec(\mathbf{a} b)\preceq\Rec(\mathbf{a} b')$;
\item[(d)] $\Rec(\mathbf{a} b') \preceq\Rec(\mathbf{a}
b b')$.
\end{longlist}
\end{lem}

\begin{pf}
Parts (a) and (b) are slightly weaker versions of
the ``Row Bumping Lemma'' in \citeauthor{Fulton1997} [(\citeyear{Fulton1997}), page 9]. The remaining
parts (c) and (d) follow using a similar argument
of comparing the ``bumping routes.''
\end{pf}

Note that part (a) [resp., part (b)] in the lemma
above implies that if a sequence $\mathbf{a}=(a_1,\ldots,a_n)$ is
arbitrary and $\mathbf{b}=(b_1,\ldots,b_k)$ is increasing (resp.,
decreasing), and $\Box_1,\ldots,\Box_{n+k}$ are the boxes of the
$\RSK$
shape associated to the concatenated sequence $\mathbf{a} \mathbf{b}$,
written in the order in which they were added (i.e., $\Box_j$ being the
box containing the entry $j$ in the recording tableau), then
$ \Box_{n+1} \prec\cdots\prec\Box_{n+k}$ (resp., $ \Box_{n+1}
\succ\cdots\succ\Box_{n+k}$).
Part (c) shows that the function $z\mapsto\Rec(\mathbf{a}
z)$ is weakly increasing with respect to the order $\preceq$.

%s2.6 #&#
\subsection{Symmetries of $\operatorname{RSK}$}

For a box $(i,j)\in\N^2$ we denote by $(i,j)^t=(j,i)$ the transpose
box, obtained under the mirror image across the axis $x=y$.
For a Young diagram $\lambda\in\partitions{n}$ the transposed diagram
$\lambda^t\in\partitions{n}$ is obtained by transposing all boxes of
the original Young diagram. In the following lemma, we recall some of
the well-known symmetry properties of the $\RSK$ algorithm.

%le2.5 #&#
\begin{lem}
\label{lem:symmetries-rsk}
Let $x_1,\ldots,x_n$ be a sequence of distinct elements and let
$\lambda
$ be the corresponding $\RSK$ shape.
Then:
\begin{longlist}[(a)]%[label=\textnormal{(\alph*)}]
\item[(a)] the $\RSK$ shape associated to the
sequence $x_n,x_{n-1},\ldots,x_1$ is equal to $\lambda^t$;
\item[(b)] the $\RSK$ shape associated to the sequence
$1-x_1,1-x_{2},\ldots,1-x_n$ is equal to $\lambda^t$;
\item[(c)] the $\RSK$ shape associated to the
sequence $1-x_n,1-x_{n-1},\ldots,1-x_1$ is equal to $\lambda$.
\end{longlist}
\end{lem}

\begin{pf}
Claim~(c) follows from (a)
and (b), which are both immediate consequences of
Greene's Theorem [\citet{Stanley1999}, Theorem A1.1.1].
\end{pf}

% \pagebreak

%s3 #&#
\section{The limit shape and the semicircle transition measure}
\label{sec:limit-shape}

%s3.1 #&#
\subsection{The limit shape of Plancherel-random diagrams}
\label{subsec:limit-shape}
In what follows, the limit shape theorem for Plancherel-distributed
random Young diagrams, due to \citet{LoganShepp1977} and \citeauthor{VervsikKerov1977} (\citeyear{VervsikKerov1977,VershikKerov1985}) (that was instrumental in the
solution of the famous Ulam problem on the asymptotics of the maximal
increasing subsequence length in a random permutation), will play a key
role, so we recall its formulation.

Given a Young diagram $\lambda=(\lambda(1),\ldots,\lambda(k))\in
\partitions{n}$, we identify it with the subregion
%
%e17 #&#
\begin{equation}
\label{eq:diag-region} A_\lambda= \bigcup_{ 1\le i\le k, 1\le j\le\lambda(i) }
[i-1,i]\times[j-1,j]
\end{equation}
of the first quadrant of the plane. Transform this region by
introducing the coordinate system
\[
u = x-y,\qquad v = x+y
\]
(the so-called Russian coordinates) rotated by 45 degrees and stretched
by the factor $\sqrt{2}$ with respect to the $(x,y)$ coordinates. In
the $(u,v)$-coordinates, the region $A_\lambda$ now has the form
\[
A_\lambda= \bigl\{ (u,v) \dvtx -\lambda'(1) \le u \le
\lambda(1), |u| \le v \le\phi_\lambda(u) \bigr\},
\]
where $\lambda'(1)=k$ is the number of parts of $\lambda$, and $\phi
_\lambda$ is a piecewise linear function on $[-\lambda'(1), \lambda
(1)]$ with slopes $\phi_\lambda' = \pm1$. We extend $\phi_\lambda$ to
be defined on all of $\R$ by setting $\phi_\lambda(u)=|u|$ for
$u\notin
[-\lambda'(1), \lambda(1)]$, as illustrated in Figure~\ref{fig:french}.
The function $\phi_\lambda$, called \emph{profile} of $\lambda$, is a
useful way to encode the shape of the diagram $\lambda$.

%f7 #&#
\begin{figure}

\includegraphics{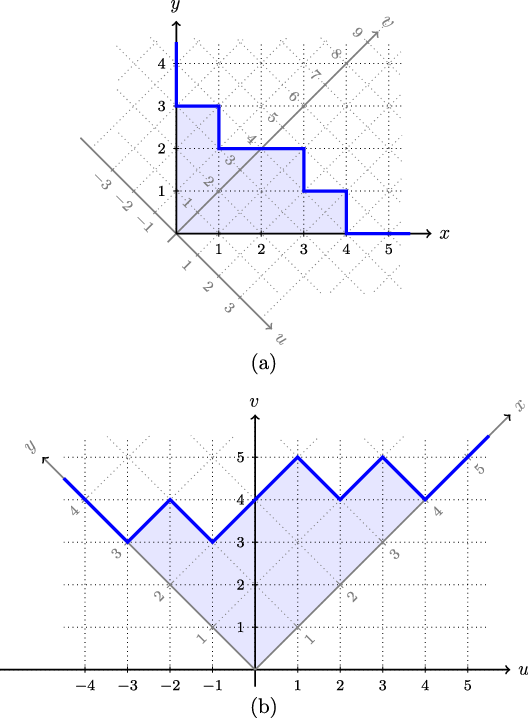}

\caption{A Young diagram $\lambda=(4,3,1)$ shown in \textup{(a)} the French, and
\textup{(b)} the Russian convention. The solid line represents the profile $\phi
_\lambda$ of the Young diagram. The coordinate system $(u,v)$
corresponding to the Russian convention and the coordinate system
$(x,y)$ corresponding to the French convention are shown.}
\label{fig:french}
\end{figure}
%
%
%}
%
%}
%

We can also consider a scaled version of $\phi_\lambda$ given by
\[
\tilde{\phi}_\lambda(u) = \frac{1}{\sqrt{n}} \phi_\lambda (\sqrt
{n} u ).
\]
This scaling leads to a diagram with constant area (equal to $2$, in
this coordinate system), and is naturally suitable for dealing with
asymptotic questions about the shape $\lambda$.

The following version of the limit shape theorem with an explicit error
estimate is a slight variation of the one given by \citet
{VershikKerov1985} [it follows from the numerical estimates in
Section~3 of that paper by modifying some parameters in an obvious way;
see also \citet{Romik2013}, Chapter~1].

%th3.1 #&#
\begin{thmm}[(The limit shape of Plancherel-random Young diagrams)]
\label{thmm-plancherel-limitshape}
Define the function $\Omega_*\dvtx \R\to[0,\infty)$ by
%
%e18 #&#
\begin{equation}
\label{eq:omegastar} \Omega_*(u) = %
\cases{\displaystyle \frac{2}{\pi} \biggl[ u
\sin^{-1} \biggl(\frac{u}{2} \biggr) +\sqrt{4-u^2}
\biggr], & \quad $\mbox{if }-\!2 \le u\le2$, \vspace *{2pt}
\cr
|u|, & \quad $\mbox{otherwise.}$}
\end{equation}
Let $\varnothing= \Lambda_0 \nearrow\Lambda_1 \nearrow\Lambda_2
\nearrow\ldots$ denote the
Plancherel growth process as in \eqref{eq:plancherel-growth}. Then
there exists a constant $C>0$ such that for any $\varepsilon>0$, we have
\[
\prob \Bigl( \sup_{u\in\R} \bigl|\tilde{\phi}_{\Lambda_n} (u ) -
\Omega_*(u) \bigr| > \varepsilon \Bigr) = O \bigl(e^{-C \sqrt
{n}} \bigr) \qquad\mbox{as }n
\to\infty.
\]
\end{thmm}

See Figure~\ref{fig-plancherel-limitshape} for an illustration of the
profile of a typical Plancherel-random diagram shown together with the
limit shape.

%f8 #&#
\begin{figure}

\includegraphics{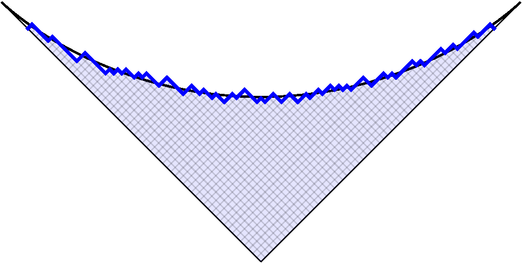}

\caption{The limit shape $v=\Omega_*(u)$ superposed with the (rescaled)
profile $\tilde{\phi}_{\Lambda_n}$ of a simulated
Plancherel-distributed random Young diagram of order $n=1000$.}
\label{fig-plancherel-limitshape}
\end{figure}

%s3.2 #&#
\subsection{The transition measure}
Next, we recall the concept of the \demph{transition measure} of Young
diagrams and its extension to smooth shapes, developed by \citeauthor{Kerov1993}
(\citeyear{Kerov1993,Kerov1999}) see also [\citet{Romik2004}]. For a Young
diagram $\lambda\in\partitions{n}$, this is defined simply as the
probability measure on the set of diagrams $\nu\in\partitions{n+1}$
such that $\lambda\nearrow\nu$ (or equivalently on the set of boxes
that can be attached to $\lambda$ to form a new Young diagram) given by
\eqref{eq:transition-rule}. Kerov observed that as a sequence of
diagrams approaches in the scaling limit a smooth shape (in a sense
similar to that of the limit shape theorem above), the transition
measures also converge, and thus depend continuously, in an appropriate
sense, on the shape. For the limit shape $\Omega_*$, which is the only
one we will need to consider, the transition measure (in this limiting
sense) is the semicircle distribution. The precise result, paraphrased
slightly to bring it to a form suitable for our application, is as follows.

%th3.2 #&#
\begin{thmm}[(Transition measure of $\plancherel_n$-random Young diagrams)]
\label{thmm-semicircle-transition-measure}
For each $n\ge1$, denote by $\mathbf{d}_n = (a_n, b_n)$ the random
position of the box that was added to the random Young diagram $\Lambda
_{n-1}$ in \eqref{eq:plancherel-growth} to obtain $\Lambda_{n}$. Then
we have the convergence in distribution
%
%e19 #&#
\begin{equation}
\label{eq:conv-dist-transition} \frac{1}{\sqrt{n}} (a_n-b_n,a_n+b_n)
\convdist ( U,V ) \qquad\mbox{as }n\to\infty,
\end{equation}
where $U$ is a random variable with the semicircle distribution $\SClaw
$ on $[-2,2]$, and $V=\Omega_*(U)$.
In other words, in the $(u,v)$-coordinates, the position of the box
added according to the transition measure \eqref{eq:transition-rule}
has in the limit a $u$-coordinate distributed according to the
semicircle distribution and its $v$-coordinate is related to its
$u$-coordinate by the function $\Omega_*$.
\end{thmm}

\begin{pf}
This follows immediately by combining Theorem~\ref
{thmm-plancherel-limitshape} with the fact that the transition measure
of the curve $\Omega_*$ is $\SClaw$, and the fact
that the mapping taking a continual Young diagram to its transition
measure is continuous in the uniform norm (with the weak topology on
measures on $\R$). For the proofs of these facts, refer to
\citeauthor{Kerov1993} (\citeyear{Kerov1993,Kerov1999})
[see also \citet{Romik2004}].
\end{pf}

%s3.3 #&#
\subsection{Weak asymptotics for the jeu de taquin path}
\label{sec-weak-asym}
As an application of these ideas, we prove the convergence in
distribution of the directions along the \jdt path in the infinite
Plancherel-random tableau. This is a weaker version of Theorem~\ref
{thmm-straight-line} that identifies the distribution \eqref
{eq:theta-dist} but does not include the fact that the \jdt path is
asymptotically a straight line. It will be convenient to work with a
modified version of the \jdt path in which time is reparameterized to
correspond more closely to the Plancherel growth process \eqref
{eq:plancherel-growth}. We call this the \demph{natural
parameterization} of the \jdt path. To define it, let $\jpathlazy_n
=\jpath_{K(n)}$ denote the position of the last box in the \jdt path
contained in the diagram $\Lambda_n$, that is, $K(n)$ is the maximal
number $k$ such that $t_{\jpath_k}$, the tableau entry in position
$\jpath_k$, is $\le n$. The reparameterized sequence $(\jpathlazy
_n)_{n\ge1}$ is simply a slowed-down or ``lazy'' version of the \jdt
path: as $n$ increases, it either jumps to its right or up if in the
Plancherel growth process a box was added in one of those two
positions, and stays put at other times.

%th3.3 #&#
\begin{thmm} \label{thmm-jdt-conv-dist}
Let $T$ be a Plancherel-random infinite Young tableau with a
naturally-parameterized \jdt path $(\jpathlazy_n)_{n=1}^\infty$. We
have the convergence in distribution
\[
\frac{\jpathlazy_n}{\Vert\jpathlazy_n\Vert} \convdist (\cos\Theta,\sin\Theta)
 \qquad\mbox{as }n\to\infty,
\]
where $\Theta$ is the random variable defined by \eqref{eq:theta-dist}.
\end{thmm}

To show this, we need the following lemma, which also gives one
possible explanation for why the slowed-down parameterization may be
considered natural (another explanation, related to the ``second-class
particle'' interpretation, is suggested in Section~\ref{sec-particle-systems}).

%le3.4 #&#
\begin{lem} \label{lem-useful-eq-dist}
For any fixed $n\geq1$, we have the equality in distribution
\[
\jpathlazy_n \equalindist\mathbf{d}_n.
\]
\end{lem}
\begin{pf}
Let $X_1,\ldots,X_n$ be i.i.d. $U(0,1)$ random variables. Let $\Lambda
_n$ be the Young diagram associated by $\RSK$ to the sequence
$(X_1,X_2,\ldots,X_n)$ and let $\widetilde{\Lambda}_{n-1}$ be the Young
diagram associated to $(X_2,X_3,\ldots,X_n)$.
From Lemmas \ref{lem-push-forward} and \ref{lem:factor-map}, we get that
\[
\jpathlazy_n \equalindist\Lambda_n \setminus\widetilde{
\Lambda }_{n-1}.
\]

Let $(Y_1,\ldots,Y_n)=(1-X_n, 1-X_{n-1}, \ldots, 1-X_1)$. In this way,
$Y_1,\ldots,Y_n$ are i.i.d. $U(0,1)$ random variables, and thus the path
in the Young graph $\varnothing=M_0 \nearrow\cdots\nearrow M_n$
corresponding to the sequence via $\RSK$
is distributed according to the Plancherel measure. It follows that
\[
\mathbf{d}_n \equalindist M_n \setminus
M_{n-1}.
\]

Applying Lemma~\ref{lem:symmetries-rsk}(c) for
the sequence $(X_1,\ldots,X_n)$ and for the sequence $(X_2,\ldots,X_n)$,
we get however that $M_n = \Lambda_n$ and $M_{n-1}=\widetilde{\Lambda
}_{n-1}$, which completes the proof.
\end{pf}

\begin{pf*}{Proof of Theorem~\ref{thmm-jdt-conv-dist}}
Define random angles $(\theta_n)_{n=1}^\infty$ by
\[
\mathbf{d}_n = (a_n,b_n)= \Vert
\mathbf{d}_n \Vert ( \cos \theta _n, \sin
\theta_n ),
\]
where $0 \le\theta_n \le\pi/2$ for $n\ge1$.
By Lemma~\ref{lem-useful-eq-dist}, it is enough to show that $\theta_n
\mathop{\rightarrow}\limits^{\mathcal{D}}\Theta$, or equivalently
that
%
%e20 #&#
\begin{equation}
\label{eq:w-lim-dist} \cot(\pi/4-\theta_n) \convdist\cot(\pi/4-\Theta) =
\frac{2}{\pi} \biggl( \sin^{-1} \biggl(\frac{W}{2} \biggr)+
\frac
{\sqrt {4-W^2}}{W} \biggr),
\end{equation}
where $W \sim\SClaw$ as in Theorem~\ref{thmm-straight-line}.
But note that
\[
\cot(\pi/4-\theta_n) = \frac{a_n+b_n}{a_n-b_n},
\]
the ratio of the $v$- and $u$- coordinates of $\mathbf{d}_n$, since the
$\pi/4$ term corresponds exactly to the angle of rotation between
$(x,y)$ and $(u,v)$ coordinates. So,
by \eqref{eq:conv-dist-transition}, $\cot(\pi/4-\theta_n) \mathop{\rightarrow}\limits^{\mathcal{D}}
V/U=\Omega_*(U)/U$, where $a_n, b_n, V$ and $U$ are defined in
Theorem~\ref{thmm-semicircle-transition-measure}, and it is easy to see
from the definition of $\Omega_*(\cdot)$ in \eqref{eq:omegastar} that
this is exactly the distribution appearing on the right-hand side
of \eqref{eq:w-lim-dist}.
\end{pf*}

Note that the proof above gives a simple geometric characterization of
the distribution of the limiting random angle $\Theta$. Namely, in the
Russian coordinate system we choose a random vector $(U,V)$ that lies
on the limit shape by drawing $U$ from the semicircle distribution
$\SClaw$, and taking $V=\Omega_*(U)$. The random variable $\Theta$ is
the angle subtended between the ray $\{u=v>0\}$ (which corresponds to
the positive $x$-axis) and the ray pointing from the origin to $(U,V)$.

%s4 #&#
\section{Plactic Littlewood--Richardson rule, Jucys--Murphy elements
and the semicircle distribution}
\label{sec:representation-voodoo}

%s4.1 #&#
\subsection{Pieri growth}
\label{subsec:pieri-growth}

Our goal in this section will be to prove a technical result that we
will need for the proofs of Theorems \ref{thmm-straight-line} and \ref
{thmm-inverse-rsk}. The result concerns a particular way of growing a
Plancherel-random Young diagram of order $n$ by $k$ additional boxes.
We refer to this type of growth as \demph{Pieri growth}, because of its
relation to the Pieri rule from algebraic combinatorics. This is
defined as follows. Fix $n, k\ge1$, and consider the following way of
generating a pair $\Lambda_n \subset\Gamma_{n+k}$ of random Young
diagrams, where $\Lambda_n \in\partitions{n}$ and $\Gamma_{n+k} \in
\partitions{n+k}$: first, take a sequence $A_1,\ldots,A_n$ of
i.i.d. random variables with the $U(0,1)$ distribution, and define
$\Lambda_n$ as the $\RSK$ shape associated with the input sequence
$A_1,\ldots,A_n$ (so, $\Lambda_n$ is distributed according to the
Plancherel measure $\plancherel_n$ of order $n$). Next, take a sequence
$B_1,\ldots,B_k$ of i.i.d. random variables with the $U(0,1)$
distribution, \emph{conditioned to be in increasing order} [i.e., the
vector $(B_1,\ldots,B_k)$ is chosen uniformly at random from the set
$\{ (b_1,\ldots,b_k) \dvtx 0\le b_1\le\cdots\le b_k \le1 \}$], then
let $\Gamma_{n+k}$ be the $\RSK$ shape associated with the concatenated
sequence $(A_1,\ldots,A_n,B_1,\ldots,B_k)$.

Let $\nu\in\partitions{k}$ be a Young diagram with $k$ boxes or, more
generally, let $\nu=\lambda\setminus\mu$ (for $\lambda\in
\partitions
{n+k}$, $\mu\in\partitions{n}$) be a skew Young diagram with $k$
boxes. Let
\[
\Box_1=(i_1,j_1),\qquad \Box_2=(i_2,j_2),\qquad
\ldots, \qquad \Box_k=(i_k,j_k)
\]
denote the positions of its boxes (arranged in some arbitrary order).
For each $1\le\ell\le k$, we will call $u_\ell= i_\ell- j_\ell$ the
\demph{$u$-coordinate} of the box $\Box_\ell$.
(In the literature, such a $u$-coordinate is usually called the \demph
{content} of $\Box_\ell$, but in order to avoid notational collisions
with the content of a box of a Young tableau, we decided not to use
this term in this meaning.)
The sequence $(u_1,\ldots,u_k)$ of the $u$-coordinates of the boxes of
$\nu$ will turn out to be very useful.

%th4.1 #&#
\begin{thmm}
\label{thmm:growth-deterministic}
For each $n,k$, let $u_1,\ldots,u_k$ be the $u$-coordinates of the boxes
of\/ $\Gamma_{n+k}\setminus\Lambda_n$, where the Pieri growth pair
$\Lambda_n \subset\Gamma_{n+k}$ is defined above.
Let $m_{n,k}$ denote the empirical measure of the $u$-coordinates
$u_1,\ldots,u_k$ (scaled by a factor of $n^{-1/2}$), given by
\[
m_{n,k} = \frac{1}{k} \sum_{\ell=1}^k
\delta_{n^{-1/2} u_\ell},
\]
where for a real number $x$, symbol $\delta_x$ denotes a delta measure
concentrated at $x$.
Let $k=k(n)$ be a sequence such that $k=o(\sqrt{n})$ as $n\to\infty$.
Then as $k\to\infty$, the random measure $m_{n,k}$ converges weakly in
probability to the semicircle distribution
$\SClaw$, and furthermore, for any $\varepsilon>0$ and any $u\in\R$ we
have the estimate
\[
\prob \bigl(\bigl |F_{m_{n,k}}(u) - \FSC(u) \bigr| > \varepsilon \bigr) = O \biggl(
\frac{1}{k}+\frac{k}{\sqrt{n}} \biggr)\qquad \mbox{as }n\to\infty,
\]
where $\FSC$ denotes the cumulative distribution function of the
semicircle distribution $\SClaw$, and $F_{m_{n,k}}$ denotes the
cumulative distribution function of $m_{n,k}$.
\end{thmm}

In order to prove this result, we will apply the ``plactic'' version of
the Littlewood--Richardson rule (Theorem~\ref{thmm:plactic-pieri-rule})
which, roughly speaking, says that the probabilistic behavior of the
$\RSK$ shape associated to a concatenation of two random sequences with
prescribed $\RSK$ shapes coincides with the probabilistic behavior of a
random irreducible component of a certain representation of the
symmetric group. In this way, the quantities describing the
probabilistic properties of the random probability measure $m_{n,k}$
can be calculated by the machinery of representation theory, and
specifically the Jucys--Murphy elements. We present the necessary tools below.

%s4.2 #&#
\subsection{The symmetric group and its representation theory}
Let $n,k\geq1$ be given. In the following, we will view $S_n$ as the
group of permutations of the set $\{1,\ldots,n\}$, $S_k$ as the group of
permutations of the set $\{n+1,\ldots,n+k\}$ and $S_{n+k}$ as the group
of permutations of $\{1,\ldots,n+k\}$. In this way $S_n\times S_k$ is
identified with the subgroup of $S_{n+k}$ consisting of those
permutations of $\{1,\ldots,n+k\}$ which leave the sets $\{1,\ldots,n\}$
and $\{n+1,\ldots,n+k\}$ invariant. In this article, we will consider
only the groups which have one of the above forms. We review below some
basic facts from representation theory, tailored for this particular setup.

For a representation $\rho\dvtx G\rightarrow\End W$ of some finite group
$G$, we define its \demph{normalized character}
\[
\chi^W(g) = \frac{\Tr\rho(g)}{(\mathrm{dimension\ of}\ W)} \qquad\mbox{for } g\in G.
\]
The group algebra $\C(G)$ can be alternatively viewed as the algebra of
functions $\{f\dvtx G\rightarrow\C\}$; as multiplication we take the
convolution of functions. For any element $f\in\C[G]$ of the group
algebra, we will denote by $\chi^W(f)$ the extension of the character
by linearity:
\[
\chi^W(f) =\sum_{g\in G} f(g)
\chi^W(g).
\]

For a modern approach to the representation theory of symmetric groups,
we refer to the monograph of
\citet{Ceccherini-SilbersteinScarabottiTolli2010}. There is a bijective
correspondence between the set of (equivalence classes of) irreducible
representations of the symmetric group $S_n$ and the set $\partitions
{n}$ of Young diagrams with $n$ boxes. We denote by $V^\lambda$ the
irreducible representation $\rho^\lambda\dvtx  S_n \rightarrow\End
V^\lambda
$ which corresponds to $\lambda\in\partitions{n}$. The dimension of
the space $V^\lambda$ is equal to $f^\lambda$, the number of standard
Young tableaux of shape $\lambda$. We use the shorthand notation $\chi
^\lambda$ for the corresponding character $\chi^{V^\lambda}$.

Two representations of the symmetric groups will play a special role in
the following. The \demph{trivial representation} $V^{\trivial}_{S_k}$
of $S_k$ is the one for which the vector space $V^{\trivial}_{S_k}$ is
one-dimensional and any group element $g\in S_k$ acts on it trivially
by identity. The corresponding character
\[
\chi^{\trivial}_{S_k}(g) = 1
\]
is constantly equal to $1$. The trivial representation is irreducible
and corresponds to the Young diagram $(k)$ which has only one row; in
other words $V^{\trivial}_{S_k}=V^{(k)}$.
The \demph{regular representation} $V^{\regular}_{S_n}$ of $S_n$ is the
one for which the vector space $V^{\regular}_{S_n}=\C(S_n)$ is just the
group algebra and the action is given by multiplication from the left.
The corresponding character
\[
\chi^{\regular}_{S_n}(g) = \delta_{e}(g)= %
\cases{1, &\quad $ \mbox{if } g=e$,\vspace*{2pt}
\cr
0, &\quad $\mbox{otherwise},$} %
\]
is equal to the delta function at the group unit.

%s4.3 #&#
\subsection{Isomorphism between \texorpdfstring{$\mathcal{C}(\mathbb{Y}_{n})$}{C(Yn)} and
\texorpdfstring{$Z\mathbb{C}(S_n)$}{ZC(Sn)}}
\label{subsec:isomorphism}
For a Young diagram $\lambda\in\partitions{n}$, we define
\[
\projection_{\lambda} = \frac{(f^\lambda)^2}{n!} \chi^\lambda.
\]
The elements $(\projection_\lambda\dvtx  \lambda\in\partitions{n})$
form a
linear basis of the center $Z\C[S_n]$ of the group algebra.
They form a commuting family of orthogonal projections, in other words
\[
\projection_\lambda\projection_\mu= %
\cases{\projection_\lambda,& \quad$\mbox{if }\lambda=\mu$, \vspace*{2pt}
\cr
0, & \quad $\mbox
{otherwise,}$} %
\]
which shows that
\[
( f\dvtx \partitions{n}\rightarrow\C ) \mapsto\sum_{\lambda
\in
\partitions{n}}
f(\lambda) \projection_\lambda\in Z\C(S_n)
\]
is an isomorphism between the commutative algebra $\mathcal
{C}(\partitions{n})$ of functions on $\partitions{n}$ (with pointwise
addition and multiplication) and the center $Z\C(S_n)$ of the symmetric
group algebra.
Thanks to this isomorphism any $f\in\mathcal{C}(\partitions{n})$ can be
identified with an element of the center $Z\C(S_n)$ which for
simplicity will be denoted by the same symbol.

The inverse isomorphism associates to $f\in Z\C(S_n)$ a function on
Young diagrams which is explicitly given by
%
%e21 #&#
\begin{equation}
\label{eq:inverse-isopmorphism} \lambda\mapsto\chi^\lambda(f).
\end{equation}

%s4.4 #&#
\subsection{The random Young diagram associated to a representation}
\label{subsec:random-young-diagram-to-representation}
For a representation $W$ of the symmetric group $S_n$ we consider its
decomposition into irreducible components:
%
%e22 #&#
\begin{equation}
\label{eq:decomposition} W= \bigoplus_{\lambda\in\partitions{n}} m_\lambda
V^\lambda,
\end{equation}
where $m_\lambda\in\N\cup\{0\}$ denotes the multiplicity.
The representation $W$ induces a probability measure on $\partitions
{n}$ given by
\[
\mathbb{P}_W(\lambda) = \frac{m_\lambda (\mathrm{dimension\ of}\ V^\lambda
)}{(\mathrm{dimension\ of}\ W)} \qquad\mbox{for } \lambda\in
\partitions{n}.
\]
In other words, the representation $W$ of $S_n$ gives rise to a random
Young diagram $\Lambda$ with $n$ boxes; we will say that $\Lambda$ is
the \demph{random Young diagram associated to the representation $W$}.
The probability of $\lambda$ is proportional to the total dimension of
all irreducible components of $W$ which are of type $[\lambda]$.
Alternatively, we can select some linear basis $e_1,\ldots,e_l$ of the
vector space $W$ in such a way that each basis vector $e_i$ belongs to
one of the summands in \eqref{eq:decomposition}. With the uniform
measure we randomly select a basis vector $e_i$; this
vector corresponds to a Young diagram $\Lambda$ which has the desired
distribution.

This choice of probability measure on $\partitions{n}$ has an advantage
that the corresponding expected value of random variables has a very
simple representation-theoretic interpretation. Namely,
for $f\in\mathcal{C}(\partitions{n})$ [which under the identification
from Section~\ref{subsec:isomorphism} can be seen as $f\in Z\C(S_n)$],
it is immediate from the definitions that
%
%e23 #&#
\begin{equation}
\label{eq:isomorphism-mean-value} \E_W f(\Lambda) = \chi^W (f),
\end{equation}
where $\E_W$ denotes the expectation with respect to the measure
$\mathbb{P}_W$.

An important example is the case when $W=V^{\regular}_{S_n}$ is the
regular representation of the symmetric group; then the corresponding
probability distribution on $\partitions{n}$ is the Plancherel measure
\eqref{eq:plancherel}.

%s4.5 #&#
\subsection{Outer product and Littlewood--Richardson coefficients}
If $V$ is a representation of $S_n$ and $W$ is a representation of
$S_k$ we denote by
\[
V\oproduct W= (V \otimes W) \uparrow_{S_n \times S_k}^{S_{n+k}}
\]
their \demph{outer product}. It is a representation of $S_{n+k}$ which
is induced from the tensor representation $V\otimes W$ of the Cartesian
product $S_n\times S_k$.

There are several equivalent ways to define Littlewood--Richardson
coefficients but for the purposes of this article it will be most
convenient to use the following one. For Young diagrams $\lambda\in
\Young{n}$, $\mu\in\Young{k}$, $\nu\in\Young{n+k}$, we define
the \demph
{Littlewood--Richardson coefficient} $c_{\lambda,\mu}^{\nu}$ as the
multiplicity of the irreducible representation $V^\lambda\otimes V^\mu$
of the group $S_n\times S_k$ in the restricted representation $V^\nu
\downarrow^{S_{n+k}}_{S_n\times S_k}$.

Equivalently, $c_{\lambda,\mu}^{\nu}$ is equal to the multiplicity of
the irreducible representation $V^\nu$ in the outer product $V^\lambda
\oproduct V^\mu$. It follows that the random Young diagram associated
to the outer product $V^\lambda\oproduct V^\mu$ has the distribution
%
%e24 #&#
\begin{equation}
\label{eq:distribution-outer-product} \mathbb{P}_{V^\lambda\oproduct V^\mu}(\nu) =
\frac{1}{\mathrm{dimension\
of\ }V^\lambda\oproduct V^\mu}
c_{\lambda,\mu}^{\nu} f^\nu.
\end{equation}

%s4.6 #&#
\subsection{The plactic Littlewood--Richardson rule}

The following result is essentially a reformulation of the usual form
of the plactic Littlewood--\break Richardson rule [\citet{Fulton1997}, Chapter~5].

%th4.2 #&#
\begin{thmm}
\label{thmm:plactic-pieri-rule}
Let the Young diagrams $\lambda\in\Young{n}$, $\mu\in\Young{k}$ be
fixed. Let $\mathbf{A}=(A_1,\ldots,A_n)\in[0,1]^n$ and $\mathbf
{B}=(B_1,\ldots,B_k)\in[0,1]^k$ be random sequences sampled according to
the product of Lebesgue measures, conditioned so that $\lambda$,
respectively $\mu$, is the $\RSK$ shape associated to $\mathbf{A}$,
respectively $\mathbf{B}$.
Then the distribution of the $\RSK$ shape associated to the
concatenated sequence $\mathbf{A} \mathbf{B}$ coincides with the
distribution \eqref{eq:distribution-outer-product} of the random Young
diagram associated to the representation $V^\lambda\oproduct V^\mu$.
\end{thmm}

\begin{pf}
Let $\A=[0,1]$ be the alphabet (linearly ordered set) of the numbers
from the unit interval. For the purpose of the following definition, we
consider $\RSK_n\dvtx \A^n\rightarrow\partitions{n}$ as a map which to words
of length $n$ associates the corresponding $\RSK$ shape. For a Young
diagram $\lambda\in\partitions{n}$, we define the formal linear
combination
\[
\strangeSchur_\lambda= \frac{n!}{(f^\lambda)^2} \mathop{\sum
_{\mathbf
{A}=(A_1,\ldots,A_n)\in\A^n, }}_{ \RSK_n(\mathbf{A})=\lambda} \mathbf{A}
\]
of all words for which the $\RSK$ shape is equal to $\lambda$. This
formal linear combination can be alternatively viewed as a function
$\strangeSchur_\lambda\dvtx  \A^n \rightarrow\R$; then it becomes a density
of a probability measure on $\A^n$. This measure is the probability
distribution of a random sequence $\mathbf{A}$ with the uniform
distribution on $\A^n$, conditioned to have the $\RSK$ shape equal to
$\lambda$.

There are $f^\lambda$ possible choices of a recording tableau of shape
$\lambda$. It follows that the plactic class corresponding to a given
insertion tableau of shape $\lambda$ consists of $f^\lambda$ elements
of $\A^n$. Therefore, the embedding of $\A^n$ into the plactic monoid
maps $\strangeSchur_\lambda$ to $\frac{n!}{f^\lambda} S_\lambda$, where
$S_\lambda$ is the plactic Schur polynomial, defined as
\[
S_\lambda=\sum_{\mathrm{shape}(P)=\lambda} P,
\]
where the sum runs over all increasing tableaux $P$ of shape $\lambda$
and with the entries in the alphabet $\A$.

We now use one of the forms of the plactic Littlewood--Richardson rule
[\citet{Fulton1997}, page 63], which says that for arbitrary $\lambda\in
\partitions{n}$, $\mu\in\partitions{k}$, we have that
\[
S_\lambda S_\mu= \sum_{\nu\in\partitions{n+k}}
c_{\lambda,\mu
}^{\nu} S_{\nu},
\]
where the product is taken in the plactic monoid. Therefore,
%
%e25 #&#
\begin{equation}
\label{eq:plactic-LR} \strangeSchur_\lambda\strangeSchur_\mu=
\frac{1}{{n+k\choose k}
f^\lambda f^\mu} \sum_{\nu\in\partitions{n+k}} c_{\lambda,\mu
}^{\nu}
f^\nu\strangeSchur_{\nu}.
\end{equation}

If we interpret $\strangeSchur_{\lambda}$ and $\strangeSchur_{\mu}$ as
densities of probability measures on $\A^n$ and $\A^k$, respectively,
and as a product we take concatenation of sequences, then
$\strangeSchur_{\lambda} \strangeSchur_{\lambda}$ can be
interpreted as
a density of a probability measure on $\A^{n+k}$. In this way, \eqref
{eq:plactic-LR} can be interpreted as follows: the left-hand side in
the plactic monoid is equal to the distribution of the $\RSK$ shape
associated to the concatenated sequence $\mathbf{A} \mathbf{B}$.
The probability distribution of this $\RSK$ shape is given by the
coefficients standing at the right-hand side:
\[
\prob \bigl(\RSK_n(\mathbf{A}\mathbf{B})=\nu \bigr) =
\frac
{1}{{n+k\choose k} f^\lambda f^\mu} c_{\lambda,\mu}^{\nu} f^\nu,
\]
which coincides with \eqref{eq:distribution-outer-product}, as required.
\end{pf}

%s4.7 #&#
\subsection{Jucys--Murphy elements and $u$-coordinates of boxes}

We define the \demph{Jucys--Murphy elements} as the elements of the
symmetric group algebra
\[
X_i = (1,i)+\cdots+(i-1,i) \in\C(S_{n})
\]
given for each $1\leq i \leq n$ by the formal sum of transpositions
interchanging the element $i$ with smaller numbers.
The following lemma summarizes some fundamental properties of
Jucys--Murphy elements [\citet{Jucys1974}].

%le4.3 #&#
\begin{lem}
\label{lem:Jucys-Murphy-1}
Let $\lambda\in\Young{n}$ be a Young diagram, and let $u_1,\ldots,u_n$
be the $u$-coordinates of its boxes. Let $P(x_1,\ldots,x_n)$ be a
symmetric polynomial in $n$ variables. Then:
\begin{longlist}[1.]
\item[1.]$P(X_1,\ldots,X_n) \in\C(S_n)$ belongs to the center of the
group algebra.
\item[2.]We denote by $\rho^\lambda\dvtx S_n\rightarrow V^\lambda$ the
irreducible representation of the symmetric group $S_n$ corresponding
to the Young diagram $\lambda$; then the operator $\rho^\lambda (
P(X_1,\ldots,X_n)  )$ is a multiple of the identity operator, and
hence can be identified with a complex number. The value of this number
is equal to
\[
\chi^\lambda \bigl( P(X_1,\ldots,X_n) \bigr) =
P(u_1,\ldots,u_n).
\]
\end{longlist}
\end{lem}

%s4.8 #&#
\subsection{Growth of Young diagrams and Jucys--Murphy elements}

This section is devoted to the proof of the following result which will
be essential for the proof of Theorem~\ref{thmm:growth-deterministic}.

%th4.4 #&#
\begin{thmm}
\label{thmm:how-to-calculate-moments}
We keep the notation from Section~\ref{subsec:pieri-growth}, except
that the $u$-coordinates of the boxes of $\Gamma_{n+k}\setminus
\Lambda
_n$ will now be denoted by $u_{n+1},\ldots,u_{n+k}$.
For any symmetric polynomial $P(x_{n+1},\ldots,x_{n+k})$ in $k$
variables we have
%
%e26 #&#
\begin{equation}
\label{eq:it-is-all-characters}\qquad \E P(u_{n+1},\ldots,u_{n+k}) = \bigl(
\chi^{\regular}_{S_n} \otimes \chi ^{\trivial}_{S_k}
\bigr) \bigl( P(X_{n+1},\ldots,X_{n+k}) \downarrow
^{S_{n+k}}_{S_n \times S_k} \bigr),
\end{equation}
where $F  \downarrow^{S_{n+k}}_{S_n\times S_k}\in\C(S_n\times S_k)$
denotes the restriction of $F\in\C(S_{n+k})$ to the subgroup
$S_n\times S_k$.
\end{thmm}

Before we do this, we show the following technical result.

%le4.5 #&#
\begin{lem}
\label{lem:jucys-murphy-skew}
Let $\lambda\in\partitions{n}$, $\mu\in\partitions{k}$ be given.
Let\/ $\Gamma$ be a random Young diagram associated to the outer
product $V^{\lambda}\oproduct V^{\mu}$ of the corresponding irreducible
representations.
Let $u_{n+1},\ldots,u_{n+k}$ be the $u$-coordinates of the boxes of the
skew Young diagram $\Gamma\setminus\lambda$
(one can show that always $\lambda\subseteq\Gamma$).
Then for any symmetric polynomial $P(x_{n+1},\ldots,x_{n+k})$ in $k$ variables
\[
\E P(u_{n+1},\ldots,u_{n+k}) = \bigl( \chi^{\lambda}
\otimes\chi ^{\mu} \bigr) \bigl( P(X_{n+1},\ldots,X_{n+k})
\downarrow^{S_{n+k}}_{S_n
\times S_k} \bigr).
\]
\end{lem}
\begin{pf}
This proof is modeled after the proof of Proposition~3.3 in \citet{Biane1998}.
The regular representation of the symmetric group decomposes as follows:
%
%e27 #&#
\begin{equation}
\label{eq:left-regular} \C( S_{n+k} ) = \bigoplus_{\gamma\in\partitions{n+k}}
V^\gamma \otimes V^\gamma
\end{equation}
as an $S_{n+k} \times S_{n+k}$-module.
The image of the projection $\projection_{\lambda} \otimes
\projection
_{\mu}\in\C(S_n\times S_k)$ acting from the left on the decomposition
\eqref{eq:left-regular} is equal to
%
%e28 #&#
\begin{eqnarray}
\label{eq:left-regular-2} (\projection_{\lambda} \otimes\projection_{\mu})
\C( S_{n+k} ) &=&\bigoplus_{\gamma\in\Young{n+k}}
c_{\lambda,\mu}^{\gamma} \bigl( V^{\lambda} \otimes V^{\mu}
\bigr) \otimes V^\gamma
\nonumber
\\[-8pt]
\\[-8pt]
\nonumber
& =&  \bigl( V^{\lambda} \otimes
V^{\mu} \bigr) \otimes\bigoplus_{\gamma\in\Young{n+k}}
c_{\lambda,\mu}^{\gamma} V^\gamma,
\end{eqnarray}
which we view as a $(S_n\times S_k) \times S_{n+k}$-module and
where the multiplicity\vspace*{1pt} $c_{\lambda,\mu}^{\gamma}\in\N\cup\{0\}$
is the
Littlewood--Richardson coefficient.
It follows that if we view \eqref{eq:left-regular-2} as a (right)
$S_{n+k}$-module, the distribution of a random Young diagram associated
to it coincides with the distribution of a random Young diagram $\Gamma
$ associated to the outer product $V^{\lambda} \oproduct V^{\mu}$.

Assume that $F\in\C(S_{n+k})$ commutes with the projection
$\projection
_{\lambda}\otimes\projection_{\mu}$ and furthermore that $F$ acts from
the left on \eqref{eq:left-regular-2} as follows: on the summand
corresponding to $\gamma\in\partitions{n+k}$ it acts by multiplication
by some scalar which we will denote by $F(\gamma)$. From the above
discussion, it follows that if $\Gamma$ is a random Young diagram
associated to the outer product $V^{\lambda} \oproduct V^{\mu}$ then
\[
\E F(\Gamma) = \frac{\Tr F }{\mathrm{(dimension\ of\ the\ image\ of}\
\projection_{\lambda} \otimes\projection_{\mu})},
\]
where for the meaning of the trace $\Tr F$ we view $F$ as acting from
the left on \eqref{eq:left-regular-2}. The numerator is equal to the
trace of $(\projection_{\lambda} \otimes\projection_{\mu}) F \in
\C
(S_{n+k})$ which we view this time as acting from the left on the
regular representation, thus it is equal to
\[
(n+k)! \bigl[(\projection_{\lambda} \otimes\projection_{\mu}) F
\bigr](e)= \frac{(n+k)!  (f^\lambda )^2  (f^\mu
)^2}{n!^2 k!^2} \bigl[ \bigl(\chi^{\lambda} \otimes
\chi^{\mu} \bigr) F \bigr](e).
\]
The last factor on the right-hand side can be written as
\begin{eqnarray*}
\bigl[ \bigl(\chi^{\lambda} \otimes\chi^{\mu} \bigr) F \bigr](e)& =&
\sum_{g\in
S_{n}\times S_{k}} \bigl(\chi^{\lambda}\otimes
\chi^{\mu} \bigr) \bigl(g^{-1} \bigr) F(g) \\
&=&
\sum_{g\in S_{n}\times S_{k}} \bigl(\chi^{\lambda}\otimes
\chi^{\mu} \bigr) (g) F(g) = \bigl(\chi^{\lambda}\otimes
\chi^{\mu} \bigr) \bigl( F \downarrow ^{S_{n+k}}_{S_n\times S_k}
\bigr),
\end{eqnarray*}
where we used the fact that the characters of the symmetric groups
satisfy $\chi^\gamma(g)=\chi^\gamma(g^{-1})$.
Thus,
\[
\E F(\Gamma) = C_{\lambda,\mu} \bigl(\chi^{\lambda} \otimes
\chi^{\mu} \bigr) \bigl( F \downarrow^{S_{n+k}}_{S_n\times S_k}
\bigr)
\]
for some constant $C_{\lambda,\mu}$ which depends only on $\lambda$ and
$\mu$. In order to calculate the exact value of this constant, we can
take $F=\delta_e\in\C(S_{n+k})$ to be the unit of the symmetric group
algebra $\C(S_{n+k})$ which therefore corresponds to a function
$F\dvtx \partitions{n}\rightarrow\C$ which is identically equal to $1$. It
follows that $C_{\lambda,\mu}=1$, and thus
%
%e29 #&#
\begin{equation}
\label{eq:gula2} \E F(\Gamma) = \bigl(\chi^{\lambda} \otimes
\chi^{\mu} \bigr) \bigl( F \downarrow^{S_{n+k}}_{S_n\times S_k}
\bigr).
\end{equation}

We denote by $p_\ell$ the power-sum symmetric polynomial
\[
p_\ell(x_{n+1},\ldots,x_{n+k}) = \sum
_{1\leq i \leq k} x_{n+i}^\ell.
\]
Let $u_{1},\ldots,u_{n}$ be the $u$-coordinates of the boxes of the
Young diagram $\lambda$.
For a given Young diagram $\gamma\in\partitions{n+k}$ such that
$\lambda
\subseteq\gamma$ we denote by $u_{n+1},\ldots,u_{n+k}$ the
$u$-coordinates of the boxes of $\gamma\setminus\lambda$; in this way
$u_1,\ldots,u_{n+k}$ are the $u$-coordinates of the boxes of $\gamma$.
Lemma~\ref{lem:Jucys-Murphy-1} shows that the operator
%
%e30 #&#
\begin{equation}
\label{eq:bubu} \sum_{1\leq i \leq n+k} X_i^\ell
\in\C(S_{n+k})
\end{equation}
acts from the right on \eqref{eq:left-regular-2} as follows: on the
summand corresponding to $\gamma$ it acts by multiplication by the
scalar $\sum_{1\leq i \leq n+k} u_i^\ell$. Furthermore, it does not
matter if we act from the left or from the right because \eqref
{eq:bubu} belongs to the center of $\C(S_{n+k})$, and thus it commutes
with the projection $\projection_{\lambda}\otimes\projection_{\mu}$.

Lemma~\ref{lem:Jucys-Murphy-1} shows that the operator
%
%e31 #&#
\begin{equation}
\label{eq:bubu2} \sum_{1\leq i \leq n} X_i^\ell
\in\C(S_n)
\end{equation}
belongs to the center of the symmetric group algebra $\C(S_n)$
therefore it commutes with the projector $\projection_{\lambda
}\otimes
\projection_{\mu} \in\C(S_n) \otimes\C(S_k) \subseteq\C(S_{n+k})$.
Furthermore, Lemma~\ref{lem:Jucys-Murphy-1} shows that \eqref{eq:bubu2}
acts from the left on \eqref{eq:left-regular-2} as follows: on any
summand it acts by multiplication by the scalar $\sum_{1\leq i \leq n}
u_i^\ell$. It follows that the difference of \eqref{eq:bubu} and~\eqref
{eq:bubu2}
\[
\sum_{1\leq i \leq n+k} X_i^\ell- \sum
_{1\leq i \leq n} X_i^\ell= \sum
_{1\leq i \leq k} X_{n+i}^\ell =p_\ell(X_{n+1},
\ldots,X_{n+k})
\]
commutes with $\projection_{\lambda}\otimes\projection_{\mu}$ and acts
on \eqref{eq:left-regular-2} from the left as follows: on the summand
corresponding to $\gamma$ it acts by multiplication by
\[
\sum_{1\leq i \leq n+k} u_{i}^\ell- \sum
_{1\leq i \leq n} u_{i}^\ell= \sum
_{1\leq i \leq k} u_{n+i}^\ell=p_\ell(u_{n+1},
\ldots,u_{n+k}).
\]

Since power-sum symmetric functions generate the algebra of symmetric
polynomials, we proved in this way that
$P(X_{n+1}, \ldots, X_{n+k})$ commutes with $\projection_{\lambda
}\otimes\projection_{\mu}$ and acts on \eqref{eq:left-regular-2} from
the left as follows: on the summand corresponding to $\gamma$ it acts
by multiplication by $P(u_{n+1}, \ldots, u_{n+k})$. This shows that
\eqref{eq:gula2} can be applied to $F=P(X_{n+1}, \ldots, X_{n+k})$ which
completes the proof.
\end{pf}

\begin{pf*}{Proof of Theorem~\ref{thmm:how-to-calculate-moments}}
The construction of Pieri growth given in Section~\ref{subsec:pieri-growth} can be formulated equivalently as follows. First,
choose a random Young diagram $\Lambda_n$ according to the Plancherel
measure of order $n$; in other words $\Lambda_n$ is a random Young
diagram with the distribution corresponding to the left regular
representation. Then, conditioned on the event $\Lambda_n=\lambda\in
\partitions{n}$, we take $(A_1,\ldots,A_n)$ to be a vector of
i.i.d. $U(0,1)$ random variables conditioned to have $\lambda$ as its
associated $\RSK$ shape; and then similarly take $(B_1,\ldots,B_k)$ to
be a vector of i.i.d. $U(0,1)$ random variables conditions to have the
single-row diagram $(k)$ as its associated $\RSK$ shape.

For $F\in\C(S_n\times S_k)$, we define
$ ( \Id\otimes\chi^{\trivial}_{S_k}  ) F \in\C(S_n)$
by a
partial application of the character $\chi^{\trivial}_{S_k}$ to the
second factor as follows:
\[
\bigl[ \bigl( \Id\otimes\chi^{\trivial}_{S_k} \bigr) F \bigr] (g)
=\sum_{h\in S_k} \chi^{\trivial}_{S_k}(h)
F(g,h) \qquad\mbox{for } g\in S_n,
\]
where we view $(g,h)\in S_n\times S_k$.

Theorem~\ref{thmm:plactic-pieri-rule} shows that if we condition over
the event $\Lambda_n=\lambda$ then the distribution of the $\RSK$ shape
associated to the concatenated sequence $(A_1,\ldots,A_n,B_1, \ldots,B_k)$
coincides with the distribution of the random Young diagram associated
to the representation $V^\lambda\oproduct V^{\trivial}_{S_k}$.
Lemma~\ref{lem:jucys-murphy-skew} shows that the conditional expected
value is given by
%
%e32 #&#
\begin{eqnarray}
\label{eq:friday}&& \E \bigl( P(u_{n+1},\ldots,u_{n+k}) |
\Lambda_n=\lambda \bigr) \nonumber\\
&&\qquad=
\bigl( \chi^{\lambda} \otimes\chi^{\trivial}_{S_k} \bigr)
\bigl( P(X_{n+1},\ldots,X_{n+k}) \downarrow^{S_{n+k}}_{S_n \times S_k}
\bigr) \\
&&\qquad=
\chi^\lambda \bigl( \bigl( \Id\otimes\chi^{\trivial}_{S_k}
\bigr) \bigl( P(X_{n+1},\ldots,X_{n+k}) \downarrow^{S_{n+k}}_{S_n \times S_k}
\bigr) \bigr).\nonumber
\end{eqnarray}
If we view it as a function of $\lambda\in\partitions{n}$, then
\eqref
{eq:inverse-isopmorphism} shows that it corresponds to the central element
%
%e33 #&#
\begin{equation}
\label{eq:auxiliary-character} \bigl( \Id\otimes\chi^{\trivial}_{S_k} \bigr)
\bigl( P(X_{n+1},\ldots,X_{n+k}) \downarrow^{S_{n+k}}_{S_n \times S_k}
\bigr)\in\C(S_n).
\end{equation}

Let us take the mean value of both sides of \eqref{eq:friday}. The mean
value of the left-hand side is equal to the left-hand side of \eqref
{eq:it-is-all-characters}. The mean of the right-hand side, by~\eqref
{eq:auxiliary-character} and \eqref{eq:isomorphism-mean-value}, is
equal to the right-hand side of \eqref{eq:it-is-all-characters}. In
this way, we showed that equality \eqref{eq:it-is-all-characters}
holds true.
\end{pf*}

%s4.9 #&#
\subsection{Moments of Jucys--Murphy elements}

For $\alpha\in\N$, we define the appropriate moment of the random
measure $m_{n,k}$:
\[
M_\alpha= M_\alpha(n,k) = \int_\R
z^\alpha \,dm_{n,k} = \frac
{1}{k} n^{-{\alpha}/{2}} \sum
_{\ell=1}^k u_\ell^\alpha.
\]
Notice that $M_\alpha$ is a random variable. In this section, we will
find the asymptotics of its first two moments: we will not only
calculate the limits but also find the speed at which these limits are
obtained since the latter is also necessary for the calculation of the
variance $\Var M_\alpha$.

Denote by
\[
\gamma_\alpha= \int z^\alpha \,d\SClaw= %
\cases{
C_{{\alpha}/{2}}, & \quad$\mbox{if $\alpha$ is even}$, \vspace*{2pt}
\cr
0, & \quad $\mbox{if
$\alpha$ is odd}$,} %
\]
the sequence of moments of the semicircle distribution, where $C_m =
\frac{1}{m+1}{2m\choose m}$ denotes the $m$th Catalan number. We will
prove the following.

%th4.6 #&#
\begin{thmm}
\label{thmm-malpha-moments}
For each $\alpha\in\mathbb{N,}$ we have
%
%e34 #&#
%e35 #&#
\begin{eqnarray}
\E M_\alpha&=& \gamma_\alpha+ O \biggl(\frac{k}{\sqrt{n}}
\biggr), \label
{eq:malpha-mean}
\\
\Var M_\alpha&=& O \biggl( \frac{1}{k} + \frac{k}{\sqrt{n}}
\biggr). \label{eq:malpha-variance}
\end{eqnarray}
\end{thmm}

This kind of calculation is not entirely new; similar calculations
already appeared in several papers [\citeauthor{Biane1995} (\citeyear
{Biane1995,Biane1998,Biane2001}, \citeauthor{Sniady2006}
(\citeyear{Sniady2006,Sniady2006a})] in the special
case $k=1$. Our calculation is not very far from the ones mentioned
above; in fact, in some aspects it is simpler than some of them since
we study a particularly simple character of the symmetric group $S_n$,
namely $\chi^{\regular}_{S_n}$ corresponding to the regular representation.

%s4.9.1 #&#
\subsubsection{The mean value of \texorpdfstring{$M_\alpha$}{Malpha}}
\label{subsec:mean-value}
Theorem~\ref{thmm:how-to-calculate-moments} shows that
%
%e36 #&#
\begin{equation}
\label{eq:moment-jm} \E M_\alpha= \frac{1}{k} n^{-{\alpha}/{2}} \bigl(
\chi^{\regular}_{S_n} \otimes\chi^{\trivial}_{S_k}
\bigr) \biggl( \sum_{1\leq i\leq k} X_{n+i}^\alpha
\downarrow^{S_{n+k}}_{S_n
\times S_k} \biggr).
\end{equation}
The problem is therefore reduced to studying the element
%
%e37 #&#
\begin{equation}
\label{eq:sum-jm} \sum_{1\leq i\leq k} X_{n+i}^\alpha=
\sum_{1\leq i\leq k} \sum_{1\leq j_1,\ldots,j_\alpha\leq n+i-1}
(n+i, j_1) \cdots(n+i, j_\alpha)\in\C(S_{n+k}).
\end{equation}

We say that $\Xi=\{ \Xi_1,\ldots, \Xi_\ell\}$ is a \demph
{set-partition} of some set $Z$ if $\Xi_1,\ldots,\Xi_\ell$ are disjoint,
nonempty subsets of $Z$ such that $\Xi_1\cup\cdots\cup\Xi_l = Z$. We
denote by $\vert\Xi \vert$ the number of parts of $\Xi$, which is
equal to
$\ell$. There is an obvious bijection between set partitions of $Z$ and
equivalence relations on $Z$.

For a given summand contributing to the right-hand side of \eqref
{eq:sum-jm}, we define the sets
\begin{eqnarray*}
Z_\Sigma& =& \bigl\{\ell\in\{1,\ldots,\alpha\}\dvtx j_\ell\leq
n \bigr\},
\\
Z_\Pi& =& \bigl\{\ell\in\{1,\ldots,\alpha\}\dvtx j_\ell\geq
n+1 \bigr\}.
\end{eqnarray*}
We also define a set-partition $\Sigma$ of the set $Z_\Sigma$ which
corresponds to the equivalence relation
\[
p\sim q \quad\iff\quad j_p=j_q\qquad \mbox{for }p,q\in
Z_\Sigma.
\]
In an analogous way, we define a set-partition $\Pi$ of the set $Z_\Pi$.

It is easy to see that if $1\leq i \leq k$, and $j_1,\ldots,j_\alpha
\leq n+i-1$, and $1\leq i'\leq k$, and $j'_1,\ldots,j'_\alpha\leq
n+i'-1$ are such that the corresponding set-partitions coincide:
$\Sigma
=\Sigma'$ and $\Pi=\Pi'$ then there exists a permutation $g\in S_n
\times S_k$ with the property that $g(n+i)=n+i'$, $g(j_\ell)=j'_{\ell
}$. It follows that the corresponding summands
\[
(n+i, j_1) \cdots(n+i, j_\alpha) \quad\mbox{and}\quad
\bigl(n+i', j'_1 \bigr) \cdots
\bigl(n+i', j'_\alpha \bigr)
\]
are conjugate by a permutation $g\in S_n \times S_k$. This implies that
the corresponding characters
\[
\bigl( \chi^{\regular}_{S_n} \otimes\chi^{\trivial}_{S_k}
\bigr) \bigl( (n+i, j_1) \cdots(n+i, j_\alpha)
\downarrow^{S_{n+k}}_{S_n
\times S_k} \bigr)
\]
are equal. This shows that we can group together summands of \eqref
{eq:sum-jm} according to the corresponding partitions $\Sigma$ and
$\Pi$.

The contribution to \eqref{eq:moment-jm} of any summand corresponding
to given set-partitions $\Pi$ and $\Sigma$ is equal to zero if $(n+i,
j_1) \cdots(n+i, j_\alpha)$ restricted to $S_n$ is not equal to the
identity for any representative $i,j_1,\ldots,j_\alpha$. Otherwise, the
total contribution of all such summands is equal to
%
%e38 #&#
\begin{equation}
\label{eq:contribution} \frac{1}{k} n^{-{\alpha}/{2}} (n)_{\vert\Sigma \vert} \biggl(
\sum_{1\leq i \leq k} (i-1)_{\vert\Pi \vert} \biggr) = O \biggl(
n^{{(2 \vert\Sigma \vert+\vert\Pi \vert-\alpha)
}/{2}} \biggl(\frac
{k}{\sqrt{n}} \biggr)^{\vert\Pi \vert} \biggr),
\end{equation}
where
\[
(m)_\ell=\underbrace{m (m-1)\cdots(m-\ell+1)}_{\ell\ \mathrm{factors}}
\]
denotes the falling factorial.

Assume that the partition $\Sigma$ has a singleton $\{\ell\}$. Then it
is easy to check that the element $j_\ell\in\{1,\ldots,n\}$ is not a
fixed point of the product $(n+i, j_1) \cdots(n+i, j_\alpha)$, hence
the contribution of such partitions $\Sigma$ is equal to zero. This
means that we can assume that every block of $\Sigma$ has at least two
elements. It follows that $2 \vert\Sigma \vert +\vert\Pi \vert
\leq
|Z_\Sigma|+
|Z_\Pi| = \alpha$. On the other hand, from the assumptions it follows
that $\frac{k}{\sqrt{n}} = o(1)$. There are the following three (not
disjoint) cases.
\begin{itemize}
\item Suppose that $2 \vert\Sigma \vert +\vert\Pi \vert \leq
\alpha$ and
$\vert\Pi \vert\geq1$. Then \eqref{eq:contribution} is equal to
\[
O \biggl( \frac{k}{\sqrt{n}} \biggr).
\]
\item Suppose that $2 \vert\Sigma \vert +\vert\Pi \vert \leq
\alpha-1$ and
$\vert\Pi \vert\geq0$. Then \eqref{eq:contribution} is equal to
\[
O \biggl( \frac{1}{n} \biggr).
\]
\item Suppose that $2 \vert\Sigma \vert +\vert\Pi \vert = \alpha
$ and $\vert\Pi \vert=0$; in other words, all blocks of $\Sigma$
have exactly two elements
and the partition $\Pi$ is empty. Then the left-hand side of \eqref
{eq:contribution} is equal to
\[
1 + O \biggl( \frac{1}{n} \biggr);
\]
this is the only case when the limit of \eqref{eq:contribution} is nonzero.
\end{itemize}
The above discussion shows that
\[
\E M_\alpha= \operatorname{Const}_\alpha+ O \biggl(
\frac{1}{n}+\frac
{k}{\sqrt{n}} \biggr)=\operatorname{Const}_\alpha+
O \biggl(\frac
{k}{\sqrt {n}} \biggr),
\]
where $\operatorname{Const}_\alpha$ is some constant which depends only
on $\alpha$.
In this way, we showed that the limit $\lim\E M_\alpha=\operatorname
{Const}_\alpha$ of \eqref{eq:moment-jm} is the same as in the simpler
case $k=1$, related to a single Jucys--Murphy element. This case was
computed explicitly by \citet{Biane1995}, who showed that
\[
\lim_{n\to\infty} n^{-{\alpha}/{2}} \chi^{\regular}_{S_{n+1}}
\bigl(X_{n+1}^\alpha \bigr) = \gamma_\alpha,
\]
which implies that $\operatorname{Const}_\alpha=\gamma_\alpha$ and
proves \eqref{eq:malpha-mean}.

%s4.9.2 #&#
\subsubsection{The second moment of \texorpdfstring{$M_\alpha$}{Malpha}}
We now calculate the second moment of the random variable $M_\alpha$.
We have that
%
%e39 #&#
\begin{equation}\qquad
\label{eq:second-moment} \E M_\alpha^2 = \frac{1}{k^2}
n^{-\alpha} \bigl( \chi^{\regular}_{S_n} \otimes
\chi^{\trivial}_{S_k} \bigr)
\biggl( \biggl( \sum_{1\leq i_1\leq k}
X_{n+i_1}^\alpha \cdot \sum_{1\leq i_2\leq k}
X_{n+i_2}^\alpha \biggr) \downarrow^{S_{n+k}}_{S_n \times S_k}
\biggr),
\end{equation}
so, similarly as in Section~\ref{subsec:mean-value}, the problem is
reduced to studying the element
\begin{eqnarray*}
&&\biggl( \sum_{1\leq i\leq k} X_{n+i}^\alpha
\biggr)^2 \\
&&\qquad=
\sum_{1\leq i_1,i_2\leq k} \sum_{1\leq j_1,\ldots,j_\alpha\leq n+i_1-1}
\sum_{1\leq j_{\alpha+1},\ldots,j_{2\alpha} \leq n+i_2-1}
(n+i_1, j_1) \cdots\\
&&\hspace*{230pt}{}\times (n+i_1,
j_{\alpha}) \\
&&\hspace*{230pt}{}\times(n+i_2, j_{\alpha+1}) \cdots\\
&&\hspace*{230pt}{}\times
(n+i_2, j_{2\alpha})\in\C(S_{n+k}).
\end{eqnarray*}
In an analogous way, we define sets $Z_\Sigma, Z_\Pi\subseteq\{
1,\ldots
,2\alpha\}$ and the corresponding partitions $\Sigma$ and $\Pi$. In
this case, however, the analysis is more difficult, which comes from
the fact that it is possible that $j_\ell=n+i_q$ for some values of
$\ell$ and~$q$. If this happens, then we say that the block of $\Pi$
which contains $\ell$ is \emph{special}. We can again group summands
according to the corresponding set-partitions $\Sigma$, $\Pi$ (and the
information about which of the blocks of $\Pi$ is special, if any).
The detailed analysis follows. Just as before, one can assume that
every block of $\Sigma$ contains at least two elements.
\begin{longlist}[\quad]
\item[\textit{Case} 1:]
$i_1=i_2$. The total contribution of the summands of this form is just
equal to $\frac{1}{k} \E M_{2\alpha}$. We already calculated the
asymptotic behavior of such expressions; it is equal to $\frac{1}{k}
\gamma_{2\alpha} + O ( \frac{1}{\sqrt{n}}  ) $.

\item[\textit{Case} 2:] $i_1 < i_2$. Here, we divide into
two subcases.

\begin{enumerate}[\quad]%[label=\textbf{Case 2\Alph*:}]
\item[\textit{Case} 2A:] there exists a special block, that is, $j_\ell=n+i_1$ for some
index $\ell$. If the contribution is nonzero, then it is nonnegative
and bounded from above by
\begin{eqnarray*}
&&\frac{1}{k^2} n^{-\alpha} (n)_{\vert\Sigma \vert} \biggl( \sum
_{1\leq i_1
< i_2 \leq k} (i_2-1)_{\vert\Pi \vert-1} \biggr)
\\
&&\qquad=O \biggl( \frac{1}{k} n^{{(2 \vert\Sigma \vert+\vert\Pi
\vert-2\alpha)}/{2}} \biggl(\frac{k}{\sqrt{n}}
\biggr)^{\vert\Pi \vert} \biggr)\\
&&\qquad = O \biggl( \frac{1}{k} \biggr).
\end{eqnarray*}

\item[\textit{Case} 2B:] there is no special block, that is, $j_1,\ldots,j_{2\alpha}$ are
all different from $i_1$ and~$i_2$. In this case, we divide into two
further subcases.

\begin{enumerate}[\qquad]%[label=\textbf{Case 2B(\roman*):}]
\item[\textit{Case} 2B(i):] $\Pi$ is not empty. If the contribution is nonzero, then it is
nonnegative and bounded from above by
\begin{eqnarray*}
&&\frac{1}{k^2} n^{-\alpha} (n)_{\vert\Sigma \vert} \biggl( \sum
_{1\leq i_1
< i_2 \leq k} (i_2-1)_{\vert\Pi \vert} \biggr)
\\
&&\qquad= O \biggl( n^{{(2 \vert\Sigma \vert+\vert\Pi \vert-2\alpha)
}/{2}} \biggl(\frac
{k}{\sqrt{n}} \biggr)^{\vert\Pi \vert}
\biggr) \\
&&\qquad= O \biggl( \frac{k}{\sqrt{n}} \biggr).
\end{eqnarray*}

\item[\textit{Case} 2B(ii):] $\Pi$ is empty. In this case, the contribution of all such
summands to~\eqref{eq:second-moment} does not depend on $i_1$ and $i_2$
and can be written as
%
%e40 #&#
\begin{equation}
\label{eq:contribution-jm2} \frac{{k\choose2}}{k^2} n^{-\alpha} \chi^{\regular}_{S_{n}}
\bigl[ \bigl( X_{n+1}^\alpha \downarrow^{S_{n+1}}_{S_n}
\bigr)^2 \bigr].
\end{equation}
From the proof of equation (5.1.2) in \citet{Biane1998}, it follows that
\begin{eqnarray*}
&&\lim_{n\to\infty} \chi^{\regular}_{S_{n}} \biggl[
\biggl(\frac
{1}{n^{\alpha/2}} X_{n+1}^\alpha
\downarrow^{S_{n+1}}_{S_n} \biggr)^2 \biggr] \\
&&\qquad= \lim
_{n\to\infty} \biggl[ \chi^{\regular}_{S_{n}} \biggl(
\frac{1}{n^{\alpha/2}} X_{n+1}^\alpha \downarrow^{S_{n+1}}_{S_n}
\biggr)^2 \biggr] = ( \gamma_\alpha )^2,
\end{eqnarray*}
and, therefore,
\[
\chi^{\regular}_{S_{n}} \biggl[ \biggl( \frac{1}{n^{\alpha/2}}
X_{n+1}^\alpha \downarrow^{S_{n+1}}_{S_n}
\biggr)^2 \biggr]= ( \gamma_\alpha )^2 + O \biggl(
\frac{1}{n} \biggr).
\]
It follows that \eqref{eq:contribution-jm2} is equal to
\[
\frac{1}{2} ( \gamma_\alpha )^2 + O \biggl(
\frac
{1}{k}+\frac
{1}{n} \biggr).
\]
\end{enumerate}
\end{enumerate}

\item[\textit{Case} 3:]
$i_1 > i_2$. This case is analogous to Case 2 above.
\end{longlist}

To summarize, we have shown that
\[
\E M_\alpha^2 = (\gamma_\alpha )^2 + O
\biggl( \frac
{1}{k} + \frac{k}{\sqrt{n}} \biggr).
\]
Since $ \Var M_\alpha= \E M_\alpha^2 - (\E M_\alpha)^2$, combining
this with \eqref{eq:malpha-mean} we get \eqref{eq:malpha-variance},
which finishes the proof of Theorem~\ref{thmm-malpha-moments}.

%s4.10 #&#
\subsection{Proof of Theorem \protect\ref{thmm:growth-deterministic}}

By Theorem~\ref{thmm-malpha-moments}, we get using Chebyshev's
inequality that for any $\varepsilon>0$ and any $\alpha\in\N$,
%
%e41 #&#
\begin{equation}
\label{eq:moments-converge} \prob \bigl( \llvert M_\alpha- \gamma_\alpha
\rrvert > \varepsilon \bigr) = O \biggl( \frac{1}{k}+\frac{k}{\sqrt{n}}
\biggr).
\end{equation}
Furthermore, for each $\varepsilon>0$ and $u\in\R$ there exists a
$\delta
>0$ and an integer $A>0$ with the property that if $m$ is a probability
measure on $\mathbb{R}$ such that its moments (up to order $A$) are
$\delta$-close to the moments of $\SClaw$ then $\llvert  F_m(u)-\FSC
(u)\rrvert <\varepsilon$. If this were not the case, then there would exist
a sequence of measures which converges in moments to $\SClaw$ but does
not converge weakly to $\SClaw$, which is not possible, since $\SClaw$
is compactly supported and therefore uniquely determined by its moments
[\citet{Durrett2010}, Section~3.3.5]. So, we get from \eqref
{eq:moments-converge} that for any $u\in\R$,
\begin{eqnarray*}
\prob \bigl( \bigl\llvert F_{m_{n,k}}(u) - \FSC(u) \bigr\rrvert >
\varepsilon \bigr) &\le&\sum_{\alpha=1}^A \prob
\bigl( |M_\alpha-\gamma_\alpha |>\delta \bigr)
\\
&= &O \biggl( \frac{1}{k}+\frac{k}{\sqrt{n}} \biggr),
\end{eqnarray*}
which proves the claim.

%s5 #&#
\section{The asymptotic determinism of $\operatorname{RSK}$ and jeu de taquin}
\label{sec-determinism}

%s5.1 #&#
\subsection{The asymptotic determinism of $\operatorname{RSK}$}
\label{sec-asym-det-rsk}
A key fact which we will need in our proof of Theorems \ref
{thmm-straight-line}, \ref{thmm-isomorphism} and \ref
{thmm-inverse-rsk}, and which is also of interest by itself, is the
following: when applying an $\RSK$ insertion step with a fixed input $z
\in[0,1]$ to an existing insertion tableau $P_n$ which is the result
of $n$ previous insertion steps involving random inputs which are drawn
independently from the uniform distribution $U(0,1)$, the macroscopic
position of the new box that is added to the $\RSK$ shape depends
asymptotically only on the number $z$ being inserted. We refer to this
phenomenon as the \emph{asymptotic determinism of $\RSK$ insertion}.
Its precise formulation is given in the following theorem, whose proof
will be our first goal in this section.

%th5.1 #&#
\begin{thmm}[(Asymptotic determinism of $\RSK$ insertion)]
\label{thmm-asymptotic-det-rsk}
Let
\[
\FSC(t)=F_{\SClaw}(t)=\frac{1}2 + \frac{1}{\pi} \biggl(
\frac{t\sqrt {4-t^2}}{4} + \sin^{-1} \biggl(\frac{t}{2} \biggr) \biggr)
\qquad(-2\le t\le2),
\]
denote as before the cumulative distribution function of the semicircle
distribution $\SClaw$.
Fix $z\in[0,1]$. For each $n\ge1$, let $(\Lambda_n, P_n, Q_n)$ be the
(random) output of the $\RSK$ algorithm applied to a sequence
$X_1,\ldots,X_n$ of i.i.d. random variables with distribution $U(0,1)$,
and let
\[
\Box_n(z) = (i_n, j_n) =
\Rec(X_1,X_2,\ldots,X_n,z)
\]
denote the random position of the new box added to the shape $\Lambda
_n$ upon applying a further insertion step with the number $z$ as the
input. Then we have the convergence in probability
\[
n^{-1/2} (i_n-j_n,i_n+j_n
) \mathop{\rightarrow }^{\mathrm{P}} \bigl(u(z), v(z) \bigr) \qquad\mbox{as }n\to
\infty,
\]
where $u(z)=\FSC^{-1}(z)$
and $v(z)=\Omega_*(u(z))$.
Moreover, for any $\varepsilon>0$,
%
%e42 #&#
\begin{equation}
\label{eq:det-rsk-estimate} \prob \bigl[ \bigl\Vert n^{-1/2} (i_n-j_n,i_n+j_n
) - \bigl(u(z), v(z) \bigr)\bigr \Vert> \varepsilon \bigr] = O \bigl(n^{-{1}/{4}}
\bigr).
\end{equation}
\end{thmm}

The asymptotic (rescaled) position of the new box as a function of $z$
is illustrated in Figure~\ref{fig:VK-insertion}.

%f9 #&#
\begin{figure}

\includegraphics{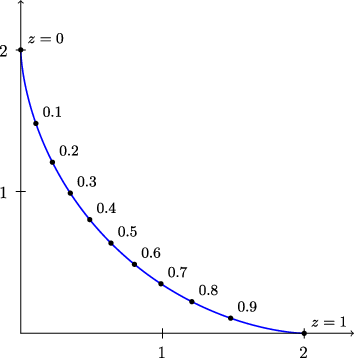}

\caption{The asymptotic position of the new box after RSK insertion as
a function of the new input~$z$.}
\label{fig:VK-insertion}
\end{figure}

\begin{pf*}{Proof of Theorem \ref{thmm-asymptotic-det-rsk}}
We consider first the case $z\in\{0,1\}$: for $z=0$ the box will be
added in the first column and for $z=1$ the box will be added in the
first row, and the question becomes equivalent to the standard problem
of finding the asymptotics\vadjust{\goodbreak} of the length of the first row and the first
column of a Plancherel-distributed random Young diagram, or
equivalently of the length of a longest increasing subsequence in a
random permutation. The large deviations results in the papers of
\citet
{DeuschelZeitouni1999} and \citet{Seppalainen1998} immediately imply
our claim in that case.

Next, fix $z\in(0,1)$ and $\varepsilon>0$.
Denote $u=\FSC^{-1}(z)\in(-2,2)$ and $u'=u+\varepsilon/4=\FSC^{-1}(z) +
\varepsilon/4$.
The cumulative distribution function $\FSC$ is strictly increasing on
$[-2,2]$; it follows that $\FSC(u') > \FSC(u) = z$. Choose some
$\Delta
>0$ in such a way that
$z+\Delta< \FSC( u') \leq1$.

Set $k=k(n) = \lceil n^{1/4}  \rceil$.
Let $\mathbf{X}=(X_1,\ldots,X_n)\in[0,1]^n$ be a sequence of
i.i.d. $U(0,1)$ random variables, and let $\mathbf{Y}=(Y_1,\ldots
,Y_k)\in
[0,1]^k$ be a sequence of i.i.d. $U(0,1)$ random variables conditioned
to be in increasing order. The $\RSK$ shapes $\Lambda_n$ and $\Gamma
_{n+k}$ associated with the sequences $\mathbf{X}$ and $\mathbf{X}
\mathbf{Y}$, respectively, are a Pieri growth pair as defined in
Section~\ref{subsec:pieri-growth}.

Denote $r=\lfloor k (z+\Delta) \rfloor\leq k$. Note that, by
interpreting the random variables $Y_1,\ldots,Y_k$ as the order
statistics of $k$ i.i.d. $U(0,1)$ random variables $Z_1,\ldots,Z_k$, we
have that
\begin{eqnarray*}
\prob ( Y_{r}<z ) &=& \prob \bigl( \mbox{at least $r$ of the
$Z_j$'s are }<z \bigr)
\\
&= &\sum_{j=r}^k \pmatrix{k
\cr
j}
z^j (1-z)^{k-j} = \prob(S_{k,z} \ge r),
\end{eqnarray*}
where $S_{k,z}$ is a random variable with a binomial distribution
$\operatorname{Binom}(k,z)$. By standard large deviations estimates, we
therefore have that for some constant $C>0$,
\[
\prob ( Y_{r}\leq z ) = O \bigl(e^{-C k} \bigr) = O
\bigl(e^{-C n^{1/4}} \bigr) \qquad\mbox{as } n\to\infty.
\]
Let $u_{n+1}\leq\cdots\leq u_{n+k}$ be the $u$-coordinates of the
boxes of $\Gamma_{n+k}\setminus\Lambda_n$ written in the order in
which they were inserted during the application of the $\RSK$ algorithm.
Assume that the event $\{ Y_r>z \}$ occurred; then parts (c)
and (d) of Lemma~\ref{lem:RSK-monotone1} imply that
\[
\Box_n= \Rec(\mathbf{X} z) \preceq\Rec(\mathbf{X} Y_1
\cdots Y_r).
\]
It follows that
\[
\frac{1}{\sqrt{n}}(i_n-j_n)=\frac{1}{\sqrt{n}}(\mbox
{$u$-coordinate of }\Box_n) \leq\frac{1}{\sqrt{n}}u_{n+r}.
\]
Now apply Theorem~\ref{thmm:growth-deterministic} to get that
\begin{eqnarray*}
&&\prob \bigl( \bigl\llvert \FSC \bigl( u' \bigr)- F_{m_{n,k}}
\bigl( u' \bigr) \bigr\rrvert > \FSC \bigl(u'
\bigr)- (z+\Delta) \bigr)
\\
&&\qquad = O \biggl( \frac{1}{k}+\frac{k}{\sqrt {n}} \biggr) = O \bigl(
n^{- 1/4} \bigr),
\end{eqnarray*}
where $m_{n,k}$ is the empirical measure of $u_{n+1}, \ldots, u_{n+k}$.
Outside of this exceptional event, we therefore have that
\[
\frac{r}{k} \leq z+\Delta\leq F_{m_{n,k}} \bigl( u'
\bigr),
\]
which, because of the meaning of the empirical measure, implies that
\[
\frac{1}{\sqrt{n}} u_{n+r} \leq u'=u+\varepsilon/4.
\]
To summarize, the above discussion shows that
%
%e43 #&#
\begin{equation}
\label{eq:one-sided-bound} \frac{1}{\sqrt{n}}(i_n-j_n) \leq u+
\varepsilon/4
\end{equation}
holds with probability $\ge1- O ( n^{- 1/4}  )$.
In order to obtain an inequality in the other direction we define
$(X'_1,\ldots,X'_n)=(1-X_1,\ldots,1-X_n)$, and let
\[
\Box'_n(z) = \bigl(i'_n,
j'_n \bigr) = \Rec \bigl(X'_1,X'_2,
\ldots,X'_n,1-z \bigr).
\]
Inequality \eqref{eq:one-sided-bound} in this setup shows that
%
%e44 #&#
\begin{equation}
\label{eq:reverse-one-sided-bound} \frac{1}{\sqrt{n}} \bigl(i'_n-j'_n
\bigr) \leq\FSC^{-1}(1-z) +\varepsilon/4
\end{equation}
holds, except on an event with probability $O ( n^{- 1/4}  )$.
But note that, first, by Lemma~\ref{lem:symmetries-rsk}(b), $(i'_n,j'_n)=(j_n,i_n)$, and second, the semicircle
distribution is symmetric, which implies that $\FSC^{-1}(1-z)=-\FSC
^{-1}(z)$. So, \eqref{eq:reverse-one-sided-bound} translates to
%
%e45 #&#
\begin{equation}
\label{eq:one-sided-bound2} \frac{1}{\sqrt{n}}(i_n-j_n) \geq u -
\varepsilon/4.
\end{equation}
Combining \eqref{eq:one-sided-bound} and \eqref{eq:one-sided-bound2},
we get that
\[
\label{eq:bound} \prob \biggl[ \biggl\llvert \frac{1}{\sqrt{n}}(i_n-j_n)
- u \biggr\rrvert > \varepsilon/4 \biggr] = O \bigl( n^{- 1/4} \bigr).
\]
Since the function $\Omega_*$ is Lipschitz with constant $1$, by
appealing to Theorem~\ref{thmm-plancherel-limitshape} we also get that
\begin{eqnarray*}
&&\prob \biggl[ \biggl| \frac{1}{\sqrt{n}}(i_n+j_n) -
\Omega_*(u)\biggr | > \varepsilon/2 \biggr]
\\
&&\qquad = O \bigl(e^{-c\sqrt{n}} \bigr) +O \bigl( n^{- 1/4} \bigr) = O \bigl(
n^{- 1/4} \bigr).
\end{eqnarray*}
These last two estimates together immediately imply \eqref
{eq:det-rsk-estimate}.
\end{pf*}

%s5.2 #&#
\subsection{Asymptotic determinism of jeu de taquin}
\label{sec-asym-det-jdt}

We now use the relationship between $\RSK$ insertion and \jdt
formulated in Lemma~\ref{lem:factor-map} to deduce from Theorem~\ref
{thmm-asymptotic-det-rsk} an analogous statement that applies to jeu de taquin,
namely the fact that prepending a fixed number $z\in[0,1]$ to $n$
i.i.d. $U(0,1)$ random inputs $X_1,\ldots,X_n$ causes the \jdt path to
exit the $\RSK$ shape at a position that is macroscopically
deterministic in the limit.
We call this property the \emph{asymptotic determinism of jeu de
taquin}, and prove it below. In the next section, we will deduce
Theorems \ref{thmm-straight-line} and \ref{thmm-inverse-rsk} from it.

%th5.2 #&#
\begin{thmm}[(Asymptotic determinism of jeu de taquin)]
\label{thmm-asym-det-jdt}
Let $(X_n)_{n=1}^\infty$ be an i.i.d. sequence of random variables with
the $U(0,1)$ distribution.
Fix $z\in[0,1]$. Let $(\jpathlazy_n(z))_{n=1}^\infty$ be the natural
parameterization of the \jdt path associated with the random infinite
Young tableau
\[
\RSK(z, X_1, X_2, \ldots),
\]
and for each $n\ge1$ denote $\jpathlazy_n(z) = (i_n,j_n)$.
Then we have the almost sure convergence
%
%e46 #&#
\begin{equation}
\label{eq:asym-det-jdt-almost-sure} n^{-1/2} (i_n-j_n,i_n+j_n
) \mathop{\rightarrow }^{{a.s.}} \bigl(-u(z), v(z) \bigr)\qquad \mbox{as }n\to
\infty,
\end{equation}
where $u(z)$ and $v(z)$ are as in Theorem~\ref{thmm-asymptotic-det-rsk}.
\end{thmm}

Note that in the setting of Theorem~\ref{thmm-asym-det-jdt} it is
possible to talk about almost sure convergence, since the random
variables are defined on a single probability space.

\begin{pf*}{Proof of Theorem \ref{thmm-asym-det-jdt}}
For each $n\ge1$, let $(\Pi_{n+1}, P_{n+1}, Q_{n+1})$ denote the
output of the $\RSK$ algorithm applied to the input sequence
$(z,X_1,\ldots,X_n)$, and let $(\Lambda_n, \widetilde{P}_n,
\widetilde
{Q}_n)$ denote the output of $\RSK$ applied to $(X_1,\ldots,X_n)$.
Lemma~\ref{lem:factor-map} shows that
\[
\widetilde{Q}_n = j(Q_{n+1}).
\]
An equivalent way of saying this is that the box $\jpathlazy_n$ is the
difference of the $\RSK$ shapes $\Pi_{n+1}$ and $\Lambda_n$.

On the other hand, let us see what happens when we reverse the sequences:
by Lemma~\ref{lem:symmetries-rsk}(a) the $\RSK$
shape associated to the sequence $(X_n,\ldots,X_1,z)$ is equal to
$(\Pi
_{n+1})^t$ and the $\RSK$ shape associated to $(X_n,\ldots,X_1)$ is
equal to $\Lambda_{n}^t$.
Therefore, the box $\jpathlazy_n^t$ (the reflection of $\jpathlazy_n$
along the principal diagonal) is the box added to the $\RSK$ shape of
$X_n,\ldots,X_1$ upon application of a further $\RSK$ insertion step
with the input $z$. This is exactly the scenario addressed in
Theorem~\ref{thmm-asymptotic-det-rsk} [except that the order of
$X_1,\ldots,X_n$ has been reversed, but that still gives a sequence of
i.i.d. $U(0,1)$ random variables]. Substituting $\jpathlazy_n^t$ for
$\mathbf{d}_n$ in that theorem, we conclude that, for any $\varepsilon>0$,
\[
\prob \bigl[ \bigl\Vert n^{-1/2} (i_n-j_n,i_n+j_n)
- \bigl(-u(z),v(z) \bigr) \bigr\Vert> \varepsilon \bigr] = O \bigl(n^{-1/4}
\bigr).
\]
This implies a weaker version of \eqref{eq:asym-det-jdt-almost-sure}
with convergence \emph{in probability}. To improve this to almost sure
convergence, we will use the Borel--Cantelli lemma, but this requires
passing to a subsequence first to get a convergent series. Setting $n_m
= m^8$, we get that
\[
\sum_{m=1}^\infty n_m^{-1/4}
\le\sum_{m=1}^\infty m^{-2} <
\infty,
\]
so from the Borel--Cantelli lemma we get that for any $\varepsilon>0$,
almost surely
\[
\bigl\Vert n_m^{-1/2} (i_{n_m}-j_{n_m},i_{n_m}+j_{n_m})
- \bigl(-u(z),v(z) \bigr) \bigr\Vert< \varepsilon\vadjust{\goodbreak}
\]
holds for all $m$ large enough. This means that we have the almost sure
convergence in \eqref{eq:asym-det-jdt-almost-sure} along the
subsequence $n=n_m$. Finally, note that $n_{m+1}/\break n_m \to1$ as \mbox{$m\to
\infty$}. It is easy to see that this, together with the fact that the
path $(\jpathlazy_{n})_n$ advances monotonically in both the $x$ and
$y$ directions, guarantees (deterministically) that convergence along
the subsequence implies convergence for the entire sequence.
\end{pf*}

%s6 #&#
\section{Proof of Theorems \protect\ref{thmm-straight-line},
\protect\ref
{thmm-isomorphism} and \protect\ref{thmm-inverse-rsk}}
\label{sec:proof-main-thms}
\mbox{}
\begin{pf*}{Proof of Theorem~\ref{thmm-straight-line}}
Let $X_1, X_2,
\ldots$ be an i.i.d. sequence of $U(0,1)$ random variables, and let
$(\jpathlazy_n)_{n=1}^\infty$ be the natural parameterization of the
\jdt path of the (Plancherel-distributed) $\RSK$ image of the sequence,
denoting as before $\jpathlazy_n=(i_n,j_n)$. Conditioning on the value
of $X_1$, the situation is exactly that of Theorem~\ref
{thmm-asym-det-jdt}. By Fubini's theorem, the almost sure convergence
in that theorem therefore implies that almost surely (even taking into
account the randomness in~$X_1$),
\[
\lim_{n\to\infty} n^{-1/2}(i_n-j_n,i_n+j_n)
= \bigl(-u(X_1), v(X_1) \bigr).
\]
It follows in particular that the limit
\[
\lim_{n\to\infty} \frac{\jpathlazy_n}{\Vert\jpathlazy_n \Vert} =: ( \cos\Theta, \sin\Theta )
\]
exists almost surely, where $\Theta$ is the random variable defined by
\[
\cot(\pi/4-\Theta) = \frac{v(X_1)}{-u(X_1)}
\]
[as in the proof of Theorem~\ref{thmm-jdt-conv-dist}, the $\pi/4$ comes
from the rotation of the $(u,v)$-coordinate system relative to the
standard one]. Since $(\jpathlazy_n)_n$ is merely a slowed-down version
of the original \jdt path $(\jpath_k)_k$, that is, $\jpathlazy
_n=\jpath
_{K(n)}$ where $K(n)\le n$ for all $n$ and $K(n)\uparrow\infty$ almost
surely as $n\to\infty$,
it follows also that
\[
\frac{\jpath_k}{\Vert\jpath_k \Vert} \mathop{\rightarrow}^{\mathrm{a.s.}} ( \cos\Theta, \sin\Theta
) \qquad\mbox{as }k\to\infty.
\]

It remains to verify that $\Theta$ has the distribution given in
\eqref
{eq:theta-dist}. This follows from Theorem~\ref{thmm-jdt-conv-dist},
which already identifies the correct distributional limit. To argue a
bit more directly, note that by the definition of $u(z)$, the random
variable $u(X_1)$ [and hence also $-u(X_1)$] is distributed according
to the semicircle distribution on $[-2,2]$, that is, it is equal in
distribution to the random variable $U$ from Theorem~\ref
{thmm-semicircle-transition-measure}. Similarly, $v(X_1)=\Omega
_*(-u(X_1))$ is equal in distribution to $V$ from that theorem. So,
$\Theta=\frac{\pi}{4}-\cot^{-1}(-v(X_1)/u(X_1))$ is equal in
distribution to $\frac{\pi}{4}-\cot^{-1}(V/U)$. This is exactly the
random variable whose distribution was shown in the proof of
Theorem~\ref{thmm-jdt-conv-dist} to be given by \eqref{eq:theta-dist}.
\end{pf*}

\begin{pf*}{Proof of Theorems \ref{thmm-isomorphism} and \ref
{thmm-inverse-rsk}}
From the discussion in Section~\ref{sec-jdt-elementary},
we know that there is a measurable subset $A\in\mathcal{B}$ of
$[0,1]^\mathbb{N}$ with $\mathrm{Leb}^{\otimes\mathbb{N}}(A)=1$ and
such that the map $\RSK
\dvtx A\to\Omega$ is defined on $A$, and satisfies the homomorphism
property \eqref{eq:factormap}. To define the inverse homomorphism, let
$B \in\mathcal{F}$ be the set
\[
B = \biggl\{ t \in\Omega \dvtx \lim_{k\to\infty} \frac{\jpath
_k(t)}{\Vert\jpath_k(t) \Vert}
\mbox{ exists } \biggr\}.
\]
This is the event in $\Omega$ on which the random variable $\Theta$
from Theorem~\ref{thmm-straight-line} is defined, and we proved that
$\plancherel(B) = 1$. Since $J$ is a measure-preserving map, the event
\[
C = \bigcap_{n=0}^\infty J^{-n}(B)
= \bigl\{ \Theta_n:= \Theta\circ J^{n-1} \mbox{ exists for
}n=1,2,\ldots \bigr\}
\]
also satisfies $\plancherel(C)=1$. On this event, we define a map
$\Delta\dvtx C\to[0,1]^\mathbb{N}$ by
\[
\Delta(t) = \bigl(F_\Theta \bigl(\Theta_1(t) \bigr),
F_\Theta \bigl(\Theta _2(t) \bigr), F_\Theta \bigl(
\Theta_3(t) \bigr), \ldots \bigr).
\]
Clearly, $\Delta$ is a measurable function since each of its
coordinates is defined in terms of the measurable functions $J$ and
$\Theta$ on $\Omega$. Now, take some sequence $\mathbf
{x}=(x_1,x_2,\ldots) \in A \cap\RSK^{-1}(C)$, and denote $t=\RSK
(\mathbf{x})$. Following the argument in the proof of Theorem~\ref
{thmm-straight-line} above, we see that $\Theta_1=\Theta_1(t)$ is
related to $x_1$ via
\[
- \cot(\pi/4-\Theta) = v(x_1)/u(x_1).
\]
In particular, $\Theta$ is a strictly increasing function of $x_1$, so,
since we also know that the measure $\mathrm{Leb}^{\otimes\mathbb
{N}}$ induces the uniform
distribution $U(0,1)$ on $x_1$ and the distribution \eqref
{eq:theta-dist} on $\Theta$, it follows that this functional relation
can be alternatively described in terms of the cumulative distribution
function of $\Theta$, namely
%
%e47 #&#
\begin{equation}
% ATTENTION TYPESETTER!!!
% This formula is not referenced from the current paper,
% however in a forthcoming paper we will need to refer to this equation.
% For this reason it would be convenient to have this formula as
% a displayed equation with a number.
x_1 =
F_\Theta(\Theta_1).
\end{equation}
Now apply the same argument to $J(t)$. By the factor property, we get
similarly that $x_2 = F_\Theta(\Theta\circ J(t)) = F_\Theta(\Theta
_2(t))$. Continuing in this way, we get that $x_n = F_\Theta(\Theta
_n(t))$ for all $n\ge1$, in other words that
\[
\Delta(t) = \mathbf{x},
\]
which shows that $\Delta$ is inverse to $\RSK$ on the set $A\cap\RSK
^{-1}(C)$. This completes the proof.
\end{pf*}

%s7 #&#
\section{Second class particles}
\label{sec-particle-systems}

In this section, we take another look at our results on \jdt on
infinite Young tableaux, this time from the perspective of the theory
of interacting particle systems. As we mentioned briefly in the
\hyperref[sec1]{Introduction}, it turns out that there is a very natural and elegant way
to reinterpret the results on the \jdt path of a random infinite Young
tableau as statements on the behavior of a \demph{second-class
particle} in a certain interacting particle system associated with the
Plancherel measure, which we call the \demph{Plancherel-TASEP} particle
system. This is not only interesting in its own right; it also draws
attention to the remarkable similarity of our results to the parallel
(and, so far, better-developed) theory of second-class particles in the TASEP.

%s7.1 #&#
\subsection{Rost's mapping} \label{sec:rost-mapping}
Before introducing the all-important concept of the second-class
particle, let us start by recalling a simpler mapping between growth
sequences of Young diagrams (i.e., paths in the Young graph) and
time-evolutions of a particle system, without the presence of a
second-class particle. To our knowledge, this mapping was first
described in the classical paper of \citet{Rost1981}. Here, the
particles occupy a subset of the sites of a lattice---usually taken
to be $\Z$, but for our purposes it will be more convenient to imagine
the particles as residing in the spaces between the lattice positions,
or equivalently on the sites of the shifted (or dual) lattice $\Z' =
\Z
+\frac{1}2$. The mapping can be described as follows: given a Young
diagram $\lambda\in\allpartitions$, draw the profile $\phi_\lambda$
of $\lambda$ in the Russian coordinate system, then project each
segment of the graph of $\phi_\lambda(u)$ where $u$ ranges over an
interval of the form $[m,m+1]$ down to the $u$-axis. In the particle
universe, a segment of slope $-1$ corresponds to the presence of a
particle at the $\Z'$ lattice site $m+\frac{1}2$, and a segment of slope
$+1$ translates to a vacant site (often referred to as a ``hole''), at
position $m+\frac{1}2$. A site containing a particle is said to be occupied.

It is now easy to see that the allowed transitions of the Young graph
(adding a box to a Young diagram $\lambda$ to get a new diagram $\nu$)
correspond to the following ``exclusion dynamics'' on the particle
system: a particle in position $m+\frac{1}2$ may jump one step to the
right to position $(m+1)+\frac{1}2$, and such a jump is only possible if
site $(m+1)+\frac{1}2$ is currently vacant. This is illustrated in
Figure~\ref{fig-diags-to-particles}. For consistency with the more
general theory of exclusion processes, we refer to these transition
rules as the \demph{TASE \textup{(}Totally Asymmetric Simple Exclusion\textup{)} rules}.

Note also that the empty Young diagram $\varnothing$ corresponds to an
initial state of the particle system wherein the positions $m+\frac{1}2$
are occupied for $m<0$ and vacant for $m\ge0$. Thus, the mapping we
described translates statements about infinite paths on the Young graph
starting from the empty diagram to statements about time-evolutions of
the particle system starting from this initial state. Note that the
mapping is purely combinatorial---we have not imposed any
probabilistic structure yet.

%f10 #&#
\begin{figure}

\includegraphics{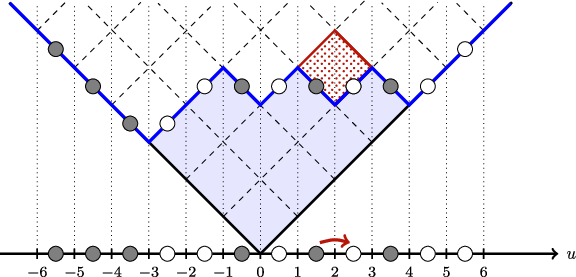}

\caption{A sample configuration of particles on the shifted lattice
$\Z
'$ corresponding via Rost's mapping to the Young diagram $(4,3,2)$.
Particles are depicted by filled circles, empty slots by empty circles.
The addition of the dotted box would correspond to a jump of one of the
particles one site to the right; such a jump can occur whenever the
site to the right of a particle is vacant.}
\label{fig-diags-to-particles}
\end{figure}

%s7.2 #&#
\subsection{Enhanced particle systems}

Next, we describe how the structure of the particle system may be
enhanced by the addition of a new kind of particle, referred to as a
second-class particle, whose behavior is different from that of both
ordinary particles and that of holes. Such a particle emerges from an
extension of Rost's mapping defined above, described by \citet
{FerrariPimentel2005}. Consider an infinite path
%
%e48 #&#
\begin{equation}
\label{eq:inf-path-young} \varnothing= \lambda_0 \nearrow
\lambda_1 \nearrow\lambda_2 \nearrow \ldots
\end{equation}
on the Young graph starting from the empty diagram. From the empty
diagram, the path always moves to the single-box diagram $\lambda_1 =
(1)$. Note that at this point, in the corresponding particle world
there is a single pair consisting of a hole lying directly to the left
of a particle. Following the terminology of Ferrari and Pimentel, we
call this pair the \emph{$*$-pair}, and call the hole on the
left side of the pair the \demph{$*$-hole} and the particle on the
right the \demph{$*$-particle}. In a picture visualizing this system,
we highlight the $*$-pair by drawing a rectangle around it; see
Figure~\ref{fig-domino-initial}. We refer to a particle configuration
with a $*$-pair as an \demph{enhanced particle configuration}, and call
the configuration corresponding to the diagram $\lambda_1$ the \demph
{initial state}.

%f11 #&#
\begin{figure}[b]

\includegraphics{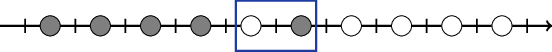}

\caption{The initial state of an enhanced particle system.}
\label{fig-domino-initial}
\end{figure}

Next, introduce dynamics to the enhanced particle system by noting that
when the $*$-particle jumps to the right, it swaps with a hole. Thus,
in such a transition a triplet of adjacent sites in a
``hole--particle--hole'' configuration (of which the leftmost two sites
represent the $*$-pair) becomes a ``hole--hole--particle'' triplet.
Following such a transition, we designate the rightmost two sites of
the triplet as the new $*$-pair. In other words, one can say that the
$*$-pair has jumped one step to the right, trading places with the hole
to its right.

Similarly, another possible transition involving a $*$-pair is when a
``particle--hole--particle'' triplet, of which the two rightmost
positions form a $*$-pair, becomes a ``hole--particle--particle'' triplet
due to the $*$-hole being jumped on by the particle to its left. In
this case, following the transition we designate the leftmost two
particles as the new $*$-pair, and say the $*$-pair jumped one step to
the left. These rules are illustrated in Figure~\ref{fig-star-pair}.

%f12 #&#
\begin{figure}

\includegraphics{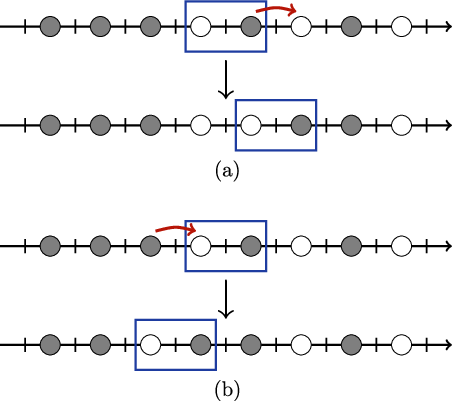}

\caption{Transitions in an enhanced particle system. Transitions not
involving the $*$-pair obey the usual TASE rules. The possible
transitions involving a $*$-pair are: \textup{(a)}
when the rightmost particle in a $*$-pair jumps to the right, the
$*$-pair also moves to the right; \textup{(b)} when
the leftmost hole in a $*$-pair is pushed to the left by a particle
jumping on it, the $*$-pair moves left.}
\label{fig-star-pair}
\end{figure}

%f13 #&#
\begin{figure}

\includegraphics{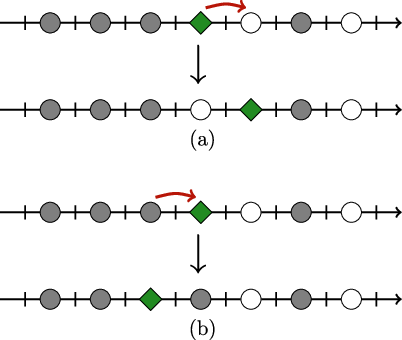}

\caption{A particle system with a second-class particle (represented as
a diamond) and the allowed transitions involving the second-class
particle. Other transitions obey the TASE rules.}
\label{fig-second-class-transitions}
\end{figure}

%
%%
%
%}
%
%
%
%}
%%
%

%s7.3 #&#
\subsection{Simplifying the $*$-pair}

We have described the strange-looking rules of evolution of a particle
system enhanced by a so-called $*$-pair. One final simplification step
will make everything much cleaner and more intuitive. Since a $*$-pair
always consists of a particle with a hole to its right, we may as well
consider the pair as occupying a single lattice position, by
contracting the two adjacent positions into one, thus effectively
``shortening'' the lattice by one unit. The result, illustrated in
Figure~\ref{fig-second-class-transitions}, is that now there are three
types of sites: those occupied by an ``ordinary'' particle, holes and a
special site representing the $*$-pair, which is of course the
second-class particle. In this context, we refer to the ordinary
particles as \demph{first-class} particles. Now the transition rules
become much more intuitive: a second-class particle (similarly to a
first-class particle) can swap with a hole to its right but not with a
first-class particle; and it can swap with a first-class particle to
its left (which we think of whimsically as the first-class particle
``pulling rank'' to overtake it, pushing it back to an inferior
position in the infinite line of particles---hence the ``class''
terminology), but not with a hole.

%s7.4 #&#
\subsection{The second-class particle and the jeu de taquin path}

Now comes a key observation, which can be understood implicitly from
the discussion in \citet{FerrariPimentel2005} by an astute reader, but
which (so far as we know) is made here explicitly for the first time.
Take an infinite sequence \eqref{eq:inf-path-young} of growing Young
diagrams, and assume that it has a recording tableau $t =
(t_{i,j})_{i,j=1}^\infty$.
Let $(\jpathlazy_n)_{n=1}^\infty$ be the \jdt path of the infinite
Young tableau $t$ given in the natural time parameterization as defined
in Section~\ref{sec-weak-asym}, and denote by $(a_n,b_n) = \jpathlazy
_n$ the coordinates of $\jpathlazy_n$.

%pr7.1 #&#
\begin{prop} \label{prop-sec-class-jdt-equiv}
For each $n\ge1$, let $u(n)$ denote the position at time $n$ of the
second-class particle in the particle system associated with the
sequence \eqref{eq:inf-path-young} via the mapping described in the
previous section, where we choose the origin of time and space such
that its initial position is $u(0)=0$.
Let $v(n)$ denote the number of times the second-class particle moved
up to time $n$.
Then we have
\[
u(n) = a_{n+1}-b_{n+1},\qquad v(n) = a_{n+1} +
b_{n+1}\qquad (n\ge 0).
\]
\end{prop}

In words, the result says that if one considers the rotated $(u,v)$
coordinate of the natural (also known as ``lazy'') parameterization of
the \jdt path, the sequence of $u$-coordinates gives the trajectory of
the second-class particle, and the $v$-coordinates parameterize the
number of jumps of the second-class particle. In particular, the
sequence of $u$-coordinates of the ordinary (nonlazy) \jdt path
$(\jpath_k)_{k=1}^\infty$ can be interpreted as the positions of the
second-class particle after its successive jumps, in a time
parameterization in which all jumps not involving the second-class
particles do not ``move the clock.''

\begin{pf*}{Proof of Proposition \ref{prop-sec-class-jdt-equiv}}
Denote $u'(n)=a_{n+1}-b_{n+1}$ and $v'(n)=a_{n+1}+b_{n+1}$. We prove by
induction on $n$ that $u(n)=u'(n)$, $v(n)=v'(n)$.
For $n=0$, we have $(u'(0),v'(0))=(u(0),v(0))=(0,0)$. For the induction
step, it is helpful to go back to the enhanced particle system picture,
and consider the position $u(n)$ of the second-class particle at time
$n$ to be the midpoint between the positions of the $*$-hole and
$*$-particle [this is compatible with the choice of origin for which
$u(0)=0$, since in the initial state the $*$-hole and $*$-particle are
at positions $\pm\frac{1}2$]. Now consider possible changes in the
vectors $(u(n),v(n))$ and $(u'(n),v'(n))$ when we increment $n$ by $1$.
For $(u(n),v(n))$, we have that
%
%e49 #&#
\begin{eqnarray}
\label{eq:three-cases-star-pair} &&\bigl(u(n+1)-u(n), v(n+1)-v(n) \bigr)
\nonumber
\\[-8pt]
\\[-8pt]
\nonumber
&&\qquad= %
\cases{ (-1,1), & \quad$\mbox{if the $*$-pair moved left at time $n$}$,
\vspace *{2pt}
\cr
(1,1), &\quad $ \mbox{if the $*$-pair moved right at time $n$}$,
\vspace *{2pt}
\cr
(0,0), &\quad $\mbox{otherwise.}$} %
\end{eqnarray}
For $(u'(n),v'(n))$, from the definition of the \jdt path it is easy to
see that
%
%e50 #&#
\begin{eqnarray}
\label{eq:three-cases-jdt}&& \bigl(u'(n+1)-u'(n),
v'(n+1)-v'(n) \bigr)
\nonumber
\\[-8pt]
\\[-8pt]
\nonumber
&&\qquad= %
\cases{ (-1,1), &\quad $\mbox{if }\lambda_{n+2}\setminus
\lambda_{n+1} = \bigl\{ (a_{n+1},b_{n+1}+1) \bigr\}$,
\vspace*{2pt}
\cr
(1,1), &\quad $\mbox{if }\lambda_{n+2}\setminus
\lambda_{n+1} = \bigl\{ (a_{n+1}+1,b_{n+1}) \bigr\}$,
\vspace*{2pt}
\cr
(0,0), &\quad $\mbox{otherwise,}$} %
\end{eqnarray}
where $\lambda_m$ denotes the $m$th Young diagram in the Young graph
path associated with the particle system [recall that time $0$ in the
enhanced particle system corresponds to the diagram $\lambda_1=(1)$,
not $\lambda_0 = \varnothing$, which explains the discrepancy in the
indices on both sides of the equation].

Finally, as Figure~\ref{fig:illustration-jdt-sec-class} illustrates, it
is easy to see that each of the three cases in~\eqref{eq:three-cases-star-pair}
is equivalent to the corresponding case in \eqref{eq:three-cases-jdt}.
Thus, we have that
\begin{eqnarray*}
&&\bigl(u(n+1)-u(n), v(n+1)-v(n) \bigr) \\
&&\qquad= \bigl(u'(n+1)-u'(n),
v'(n+1)-v'(n) \bigr),
\end{eqnarray*}
which is just what was needed to complete the induction.\vadjust{\goodbreak}
\end{pf*}

%f14 #&#
\begin{figure}

\includegraphics{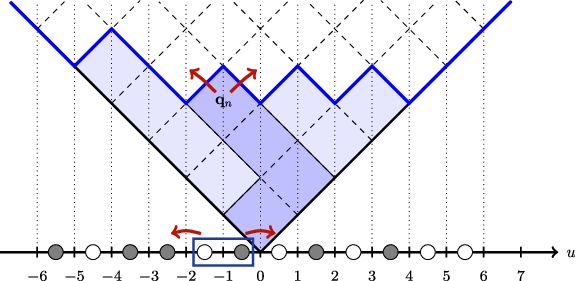}

\caption{The allowed transitions of the $*$-pair in the enhanced
particle system and the corresponding effect on the associated Young
diagram. A move of the $*$-pair to the left or right corresponds to a
north-west or north-east step, respectively, of the \jdt path in
Russian coordinates.}
\label{fig:illustration-jdt-sec-class}
\end{figure}

%s7.5 #&#
\subsection{Stochastic models}
\label{sec:stochastic}

We are finally ready to consider \emph{probabilistic} rules for the
evolution of a particle system equipped with a second-class particle as
described above. Thanks to the mapping taking a path on the Young graph
to such a system, it is enough to specify the rules of evolution for a
randomly growing family of Young diagrams.

%s7.5.1 #&#
\subsubsection{The Plancherel-TASEP}
Naturally, the first rule we consider is the particle system associated
with the Plancherel growth process (or, equivalently, with the
Plancherel measure). We call this the \demph{Plancherel-TASEP particle
system}; it is the process mentioned in Theorem~\ref
{thmm-second-class-plancherel-tasep} which we formulated in the
\hyperref[sec1]{Introduction} without explaining its precise meaning. Finally, we are in
a position to prove it.

\begin{pf*}{Proof of Theorem~\ref{thmm-second-class-plancherel-tasep}}
By Proposition~\ref{prop-sec-class-jdt-equiv}, the random variable
$X(n)$ in the theorem is simply the $u$-coordinate $a_{n+1}-b_{n+1}$ of
the natural parameterization $\jpathlazy_n = (a_n,b_n)$ of the \jdt
path of a random infinite Young tableau chosen according to Plancherel
measure. In the proof of Theorem~\ref{thmm-straight-line} in
Section~\ref{sec:proof-main-thms}, we already saw that after scaling by
a factor of $n^{-1/2}$, this random variable converges a.s. to a
limiting random variable $W$ having the semicircle distribution. This
was exactly the claim to prove.
\end{pf*}

%s7.5.2 #&#
\subsubsection{The TASEP}

A second natural and much-studied process is the \demph{Totally
Asymmetric Simple Exclusion Processes}, or \demph{TASEP}, introduced by
\citet{Spitzer1970} [this is a special case of the much wider family of
\emph{exclusion processes}, and we also consider here the TASEP itself
with only a specific initial state. For the general theory of such
processes, see \citet{Liggett1985}].
Here, we consider the simple random walk on the Young graph starting
from the empty diagram $\varnothing$. It is useful to let time flow
continuously, so the random walk is a process $(\Pi_t)_{t\ge0}$ taking
values in the Young graph $\allpartitions$, such that, given the state
of the walk $\Pi_t = \lambda$ at time $t$, at subsequent times the walk
randomly transitions to each state $\nu$ with $\lambda\nearrow\nu$ at
an exponential rate of $1$. Equivalently, each box in position $(i,j)$
gets added to the randomly growing diagram with an exponential rate of
$1$, as soon as both the boxes in positions $(i-1,j)$ and $(i,j-1)$ are
already included in the shape (where each of these conditions is
considered to be satisfied if $i=1$ or $j=1$, resp.).
This random walk is usually referred to as the \demph{corner growth model}.

A fundamental result for the corner growth model is the following limit
shape result, proved by \citet{Rost1981}, which is the analogue of
Theorem~\ref{thmm-plancherel-limitshape} for this model.

%th7.2 #&#
\begin{thmm}[(The limit shape of the corner growth model)]
Let $A_t = A_{\Pi_t}$ be the planar region associated with the random
diagram $\Pi_t$ as in \eqref{eq:diag-region}. Define
\[
L = \bigl\{ (x,y)\in[0,\infty)^2 \dvtx \sqrt{x}+\sqrt{y}\le1 \bigr
\}.
\]
Then for any $\varepsilon>0$ we have that
\[
\prob \bigl[ (1-\varepsilon)L \subseteq t^{-1} A_t
\subseteq(1+\varepsilon ) L \bigr] \to1\qquad \mbox{as }t\to\infty.
\]
\end{thmm}

%f15 #&#
\begin{figure}

\includegraphics{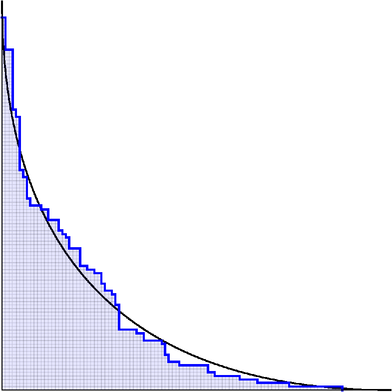}

\caption{A rescaled random Young diagram in the corner growth model and
its limit shape. The curved boundary of the limit shape is a rotated
parabola, given by the equation $\sqrt{x}+\sqrt{y}=1\ (0\le x,y\le1)$.
In Russian coordinates, it has the equation $v = \frac{1}2 (1+u^2),  |u|\le1$.}
\label{fig-cgm-limitshape}
\end{figure}

See Figure \ref{fig-cgm-limitshape} for an illustration of this result.\vadjust{\goodbreak}

Next, we can define the TASEP (without a second-class particle) as the
continuous-time interacting particle system associated with the corner
growth model via Rost's mapping described in Section~\ref{sec:rost-mapping}. This particle system follows the combinatorial TASE
rules described above for the valid particle transitions, but now in
addition the probabilistic dynamics governing these transitions are
very intuitive rules, namely that each particle can be thought of as
having a Poisson ``clock'' (independent of all others) of times during
which it attempts to jump to the right, succeeding if and only if the
space to its right is vacant. In other words, in probabilistic language
we will say that the resulting process is a Markov process with an
infinitesimal generator that can be explicitly written and encapsulates
this intuitive interpretation.

Finally, if we add the second-class particle by considering the
``enhanced'' version of Rost's mapping, we get a richer system
following the TASE rules with the additional rules governing
transitions involving the second-class particle. And again, the
probabilistic laws governing these transitions can be described in the
language of Markov processes, or equivalently in terms of each of the
first- and second-class particles having a Poisson process of times
during which it will attempt to jump.

The following result, proved by \citet{MountfordGuiol2005}, is a
precise analogue for the TASEP of Theorem~\ref
{thmm-second-class-plancherel-tasep}, and puts our own result in an
interesting context.

%th7.3 #&#
\begin{thmm} \label{thmm-second-class-tasep}
For $t\ge0$, let $X(t)$ denote the location at time $t$ of the
second-class particle in the TASEP with the initial conditions
described above. As $t\to\infty$, the trajectory of the second-class
particle converges almost surely to a straight line with a random
speed. More precisely, the limit
\[
U = \lim_{t\to\infty} \frac{X(t)}{t}
\]
exists almost surely and is a random variable distributed according to
the uniform distribution $U(-1,1)$.
\end{thmm}

A weaker version of Theorem~\ref{thmm-second-class-tasep} was proved
earlier by \citet{FerrariKipnis1995}. It is also worth noting here that
the study of trajectories of second-class particles in the TASEP, and
some of their higher-order generalizations (e.g., third-class,
fourth-class particles, etc.) in the process known as the \demph
{multi-species TASEP}, is an active field that has brought to light
very interesting results in the last few years; see the recent works by
\citet{AmirAngelValko2011,AngelHolroydRomik2009,
FerrariGonccalvesMartin2009,FerrariPimentel2005}. The paper by \citet
{AngelHolroydRomikVirag2007} also studies particle trajectories in the
\demph{uniformly random sorting network}, which is an interacting
particle system induced by a natural probability measure on Young
tableaux that shares some characteristics with Plancherel measure
(e.g., the semicircle distribution plays a special role in that context
as well). Angel et al. make detailed conjectural predictions about the
asymptotic behavior of particle trajectories in that model. It would be
interesting to see if some of the techniques used in the current paper
may be applicable to the study of these conjectures.

Second-class particles have also been studied recently in connection
with \demph{Hammersley's process}, an interacting particle system that
is also related to the $\RSK$ algorithm and Ulam's problem on longest
increasing subsequences. In this setting, a result on the trajectory of
second-class particles analogous to Theorem~\ref
{thmm-second-class-tasep} was proved by \citet{CollettiPimentel07}; see
also
\citeauthor{CatorGroeneboom05} (\citeyear{CatorGroeneboom05,CatorGroeneboom06}), \citet{CatorDobrynin06} for
related results, and \citet{CatorPimentel11} for a recent work
considerably generalizing the results of Coletti and Pimentel.

As a final note on the analogy between Theorem~\ref
{thmm-second-class-plancherel-tasep} and Theorem~\ref
{thmm-second-class-tasep}, we remark that the time parameterization of
the Plancherel-TASEP process is somewhat unnatural from the point of
view of tracking the second-class particle, and this is what accounts
for the scaling $n^{1/2}$ in Theorem~\ref
{thmm-second-class-plancherel-tasep}, which causes the second-class
particle to appear to slow down over time. As we mentioned briefly in
the \hyperref[sec1]{Introduction}, one can argue that it makes more sense to replace the
time parameter $t$ in \eqref{eq:second-class-rescaled} by $t^2$,
leading to particle system dynamics in which changes occur at a
constant time scale in each microscopic region (including in the
vicinity of the second-class particle). With such a parameterization,
the intuitive meaning of Theorem~\ref
{thmm-second-class-plancherel-tasep} becomes more similar to that of
Theorem~\ref{thmm-second-class-tasep}, namely that the second-class
particle moves asymptotically with a limiting speed, which is a random
variable whose distribution can be computed [i.e., $U(-1,1)$ in the
case of the TASEP; $\SClaw$ in the case of the Plancherel-TASEP].

%s7.6 #&#
\subsection{Competition interfaces in the corner growth model}

In the previous subsections, we reinterpreted the results on the \jdt
path of a Plancherel-random infinite Young tableau in terms of the
second-class particle in the Plancherel-TASEP particle system. One can
also go in the opposite direction, taking Theorem~\ref
{thmm-second-class-tasep} above on the behavior of a second-class
particle in the TASEP and reformulating it in the language of the
corner growth model, or equivalently, infinite Young tableaux. Indeed,
such a reformulation of Theorem~\ref{thmm-second-class-tasep} is the
central idea in the paper by \citet{FerrariPimentel2005}. While the
authors of that work do not mention Young tableaux and apparently did
not notice the connection to the \jdt path, made explicit in
Theorem~\ref{prop-sec-class-jdt-equiv} above, they phrased the result
in terms of what they call the \emph{competition interface}, which is
the boundary separating two competing growth regions in the corner
growth model. It is worth recalling this concept, which is interesting
in its own right, and noting how it interacts with our point of view.

The idea is as follows. Thinking of the diagram $\Pi_t$ as a collection
of boxes (each represented as a position in $\mathbb{N}^2$), we
decompose it into the box $(1,1)$ (assuming $t$ is large enough so that
$\Pi_t\neq\varnothing$) together with a union of boxes of two colors
\[
\Pi_t = \bigl\{(1,1) \bigr\} \cup\Pi_t^{\mathrm{green}}
\cup\Pi_t^{\mathrm{red}},
\]
so that the planar region $A_t$ associated to $\Pi_t$ is also
decomposed into a union of the regions
\begin{eqnarray*}
A_t &=& \bigl( [0,1]\times[0,1] \bigr) \cup
\biggl( \bigcup_{(i,j)\in\Pi_t^{\mathrm{green}}} [i-1,i]\times [j-1,j]
\biggr)\\
&&{} \cup \biggl( \bigcup_{(i,j)\in\Pi_t^{\mathrm{red}}} [i-1,i]
\times[j-1,j] \biggr)
\\
&=:& [0,1]^2 \cup A_t^{\mathrm{green}} \cup
A_t^{\mathrm{red}}.
\end{eqnarray*}
The color of a box $(i,j)\in\Pi_t$ is determined as follows: when the
box is added to the randomly growing Young diagram, it is classified as
green if $i=1$, red if $j=1$ [except the box $(i,j)=(1,1)$ which has no
color]; or, if $i,j\ge2$ it gets the color of that box among the two
boxes $(i,j-1)$, $(i-1,j)$ which was added to the Young diagram at the
\emph{later} time. One can think of two competing infections
propagating through the first quadrant of the plane, where a box
$(i,j)$ becomes infected at an exponential rate $1$ after the boxes
below it and to its left are already infected; once it is infected the
type of the infection (green or red) is decided according to which of
the two ``infecting'' boxes has been infected more recently than the other.

The \emph{competition interface} is defined as the boundary line
separating the green and red regions $A^t_{\mathrm{green}}$ and
$A^t_{\mathrm{red}}$; see Figure~\ref{fig-competition-interface}. As $t$
increases, this line grows by adding straight line segments in the
directions $(1,0)$ and $(0,1)$. In fact, the nature of this line is
made clear in the following result.

%f16 #&#
\begin{figure}

\includegraphics{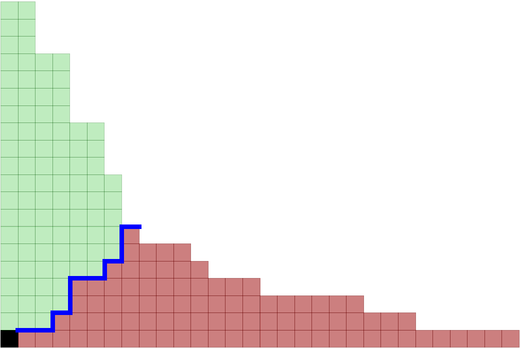}

\caption{Red and green infection regions in the corner growth model and
the competition interface.}
\label{fig-competition-interface}
\end{figure}

%pr7.4 #&#
\begin{prop} Let $\varnothing=\nu_0\nearrow\nu_1\nearrow\nu
_2\nearrow
\ldots$ denote the sequence of Young diagrams that the corner growth
model $(\Pi_t)_{t\ge0}$ passes through.
The competition interface is the polygonal line connecting the sequence
of vertices
\[
(1,1)=\jpath_1, \jpath_2, \jpath_3, \ldots
\]
given by the \jdt path box positions $(\jpath_k)_{k=1}^\infty$
associated with the Young graph path $(\nu_n)_{n=0}^\infty$.
\end{prop}

\begin{pf}
If at time $t$ the top-right endpoint of the competition interface is
in position $(a_t,b_t)$, that means that the Young diagram box indexed
by $\mathbb{N}^2$-coordinates $(a_t,b_t)$ is in $\Pi_t$ but the boxes
indexed by $(a_t+1,b_t)$ and $(a_t,b_t+1)$ are not in $\Pi_t$ (see
Figure~\ref{fig-competition-interface} for an example). Assume by
induction on $k=a_t+b_t-1$ that $\jpath_k = (a_t,b_t)$. The next step
$\jpath_{k+1}-\jpath_k$ taken by the jeu de taquin path will be $(1,0)$
or $(0,1)$ depending on which of the two boxes $(a_t+1,b_t)$ or
$(a_t,b_t+1)$ will be added to $\Pi_t$ next; it is easy to see from the
definition of the competition interface that its next step will be
determined in exactly the same way.
\end{pf}

The analogue of our Theorem~\ref{thmm-straight-line} for the corner
growth model is the following result, which is Ferrari and Pimentel's
reformulation of Theorem~\ref{thmm-second-class-tasep} in the language
of competition interfaces (which, as we observe above, is equivalent to
jeu de taquin).

%th7.5 #&#
\begin{thmm}[(Asymptotic behavior of the competition interface)]
The competition interface in the corner growth model converges to a
straight line with a random direction. More precisely, the limit
\[
(\cos\Phi, \sin\Phi) = \lim_{k\to\infty} \frac{\jpath_k}{\Vert
\jpath
_k\Vert}
\]
exists almost surely. The asymptotic angle $\Phi$ of the competition
interface is an absolutely continuous random variable, with distribution
\[
\prob(\Phi\le x) = \frac{\sqrt{\sin x}}{\sqrt{\sin x} + \sqrt {\cos
x}}\qquad (0\le x\le\pi/2).
\]
\end{thmm}

Figure~\ref{fig-density-comparison} shows a comparison of the density
function of $\Phi$ with that of $\Theta$, the asymptotic angle of the
\jdt path of a Plancherel-random infinite Young tableau.

%f17 #&#
\begin{figure}

\includegraphics{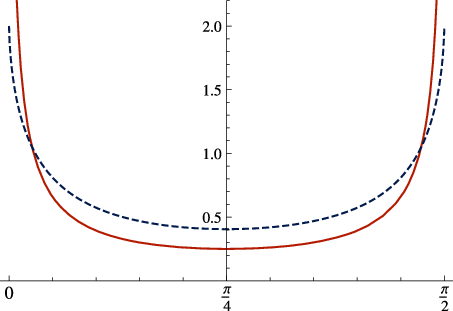}

\caption{A comparison of the density functions of $\Theta$, the
asymptotic angle of the \jdt path in a Plancherel-random infinite Young
tableau (dashed, dark blue line) and of $\Phi$ (full stroke line, in
red), the asymptotic angle of the competition interface in the corner
growth model, which can also be interpreted as a jeu de taquin path.
The density of $\Phi$ is unbounded near $0$ and $\pi/2$.}\vspace*{-5pt}
\label{fig-density-comparison}
\end{figure}

%s7.7 #&#
\subsection{Summary}

In the discussion above, we showed that the \jdt path arises naturally
in probabilistic settings which have not been noticed so far and which
go beyond its traditional applications to algebraic combinatorics,
namely the study of trajectories of second-class particle in
interacting particle systems and of the competition interface between
two randomly growing regions in the corner growth model. We hope that
the reader is convinced that the interplay between the different
interpretations and points of view is quite stimulating, and worthy of
further study.

%s8 #&#
\section{Additional directions}

\label{sec-additional-directions}

%s8.1 #&#
\subsection{Asymptotic determinism of RSK and the limit shape of the
bumping routes}

\label{sec-bumping-routes}

In a follow-up paper [\citet{RomikSniady2013}], we apply Theorem~\ref
{thmm-asymptotic-det-rsk} to prove an additional ``asymptotic
determinism'' property with more detailed information on the behavior
of RSK insertion in the random setting considered in this paper;
namely, we show that the ``bumping route'' when a deterministic input
$z$ is inserted into the insertion tableau $P_n$ (in the notation of
Theorem~\ref{thmm-asymptotic-det-rsk}) converges in the macroscopic
scaling to a limiting shape that depends only on $z$ and is given by an
explicit formula.

%s8.2 #&#
\subsection{$\operatorname{RSK}$ and random words in other alphabets}
\label{subsec-random-words}

It is natural to study the properties of $\RSK$ applied to an infinite
sequence $X_1,X_2,\ldots$
of i.i.d. random letters in a more general setup than the one
considered in the current paper,
that is, with the distribution of the letters being arbitrary.
The simplest example is the one in which $X_1,X_2,\ldots$ take values in
a finite set $[d]=\{1,\ldots,d\}$. In this case the random words and the
corresponding recording tableaux can be viewed as random walks in $\Z
^d$. \citet{OConnell2003} has shown that, under this identification,
$\RSK$ coincides with the \emph{generalized Pitman transform}
introduced by
\citet{OConnellYor2002}. The counterparts of some of the results of
the current paper have has been proved for the Pitman transform. This
topic is studied in a broader context, which also reveals interesting
connections with the representation theory of the infinite symmetric
group, in another follow-up paper by the second-named author [\citet
{RomikSniady2013a}].

\section*{Acknowledgements}

The authors are grateful to the anonymous referees for remarks and
suggestions that helped improve the paper.

%
% imsref loaded by akundreckaite, 2014-02-03 16:34:12


%suskaldyti doi

\printaddresses

\end{document}